\documentclass[a4paper,11pt]{amsart}
\usepackage{graphicx}
\usepackage[T1]{fontenc} 
\usepackage[colorlinks]{hyperref}
\usepackage{amssymb}

\title{The category of reduced orbifolds in local charts}
\author[Anke Pohl]{Anke D.\@ Pohl}
\address{Mathematisches Institut, Georg-August-Universit\"at G\"ottingen,  Bunsenstr. 3-5, 37073 G\"ottingen}
\email{anke.pohl@math.uni-goettingen.de}

\newcommand{\apref}[3]{{#1\ref{#2}#3}}
\usepackage{enumerate}
\usepackage[all]{xy}

\theoremstyle{plain}
\newtheorem{prop}{Proposition}[section]
\newtheorem{lemma}[prop]{Lemma}
\newtheorem{lemmadefi}[prop]{Lemma and Definition}
\newtheorem{thm}[prop]{Theorem}

\theoremstyle{definition}

\newtheorem{defi}[prop]{Definition}
\newtheorem{defirem}[prop]{Definition and Remark}

\newtheorem{example}[prop]{Example}

\newtheorem{construction}[prop]{Construction}

\newtheorem{specialcase}[prop]{Special Case}

\theoremstyle{remark}
\newtheorem{remark}[prop]{Remark}

\newcommand{\co}{\colon}
\DeclareMathOperator{\Germ}{Germ}
\DeclareMathOperator{\germ}{germ}
\DeclareMathOperator{\dom}{dom}
\DeclareMathOperator{\cod}{cod}
\DeclareMathOperator{\Diff}{Diff}
\DeclareMathOperator{\Orbmap}{Orb}
\newcommand{\pullback}[2]{\,{}_{#1}\!\!\times_{#2}}
\DeclareMathOperator{\Agr}{Agr}
\DeclareMathOperator{\Pgr}{Pgr}

\DeclareMathOperator{\Hom}{Hom}
\DeclareMathOperator{\Morph}{Morph}
\DeclareMathOperator{\pr}{pr}

\newcommand\N{\mathbb{N}}
\newcommand\R{\mathbb{R}}
\newcommand\C{\mathbb{C}}
\newcommand{\mc}[1]{\mathcal #1}
\newcommand{\wt}{\widetilde}
\newcommand{\eps}{\varepsilon}
\DeclareMathOperator{\id}{id}

\newcommand{\sceq}{\mathrel{\mathop:}=}
\newcommand{\seqc}{\mathrel{=\mkern-4.5mu{\mathop:}}}

\newcommand\ie{\mbox{i.\,e., }}
\newcommand\eg{\mbox{e.\,g., }}
\newcommand\wrt{\mbox{w.\,r.\,t.\@ }}

\newcommand\resp{\mbox{resp.\@}}

\begin{document}

\begin{abstract}
It is well-known that reduced smooth orbifolds and proper effective foliation Lie groupoids form equivalent categories. However, for certain recent lines of research, equivalence of categories is not sufficient. We propose a notion of maps between reduced smooth orbifolds and a definition of a category in terms of marked proper effective \'etale Lie groupoids such that the arising category of orbifolds is isomorphic (not only equivalent) to this groupoid category.
\end{abstract}

\maketitle


\section{Introduction}

Given a reduced orbifold (in local charts) and an orbifold atlas representing its orbifold structure it is well known how to construct (in an explicit way) a proper effective foliation groupoid (orbifold groupoid) from these data (see, \eg Haefliger \cite{Haefliger_orbifold} or the book by Moerdijk and Mr\v{c}un \cite{Moerdijk_Mrcun}). Over the years various authors (in particular, Moerdijk \cite{Moerdijk_survey}, Pronk~\cite{Pronk}) used this link to provide a definition of a category of orbifolds by proposing a definition of a category of orbifold groupoids, either as a 2-category or as a bicategory of fractions. Lerman \cite{Lerman_survey} provides a very good discussion of these approaches. These approaches have in common that the morphisms in the orbifold category are only given implicitly. 

While the notion of isomorphisms between orbifolds is clear, that of more general maps allows variations. Various propositions for orbifold maps in local charts have been made serving different purposes, \eg by Borzellino and Brunsden \cite{Borzellino_Brunsden} or the Chen-Ruan good maps \cite{Chen_Ruan}. Unfortunately, neither of these definitions model the morphisms coming from the groupoid category. We remark that Lupercio and Uribe \cite{Lupercio_Uribe}, even though widely believed, do not prove that Chen-Ruan good maps correspond to groupoid homomorphisms (a counterexample to the characterization of groupoid homomorphisms via good maps and thus to \cite[Proposition 5.1.7]{Lupercio_Uribe} is provided by the unique map between any orbifold with at least two charts in any orbifold atlas and the 0-dimensional connected reduced orbifold).

All these proposed groupoid categories are just equivalent, not isomorphic, to the orbifold category. This is caused by the fact that the construction mentioned above assigns the same groupoid to different (but isomorphic) orbifolds,
 and conversely various (Morita equivalent) groupoids to the same orbifold. Moerdijk and Pronk \cite{Moerdijk_Pronk} show that isomorphism classes of orbifolds correspond to Morita equivalence classes of orbifold groupoids. For many investigations about orbifolds, an equivalence of categories suffices to translate the problem to groupoids. However, one cannot investigate e.g.\@ the diffeomorphism group of an orbifold using any of the groupoid categories.

In this article we provide a correct characterization of groupoid homomorphisms in local charts. We use the arising maps to define a geometrically motivated notion of orbifold maps as certain equivalence classes. This allows us to define a geometrically natural orbifold category (with orbifolds as objects). Characterizing these orbifold maps in terms of groupoid homomorphisms enables us to define a category in terms of marked proper effective \'etale Lie groupoids (which is not the classical one) which is isomorphic to the orbifold category.

The results in this article are used by Schmeding \cite{Schmeding} to show that the diffeomorphism group of a paracompact reduced orbifold can be endowed with the structure of an (infinite-dimensional) Lie group. Moreover,  this Lie group is even $C^k$-regular for any $k\in \N_0\cup\{\infty\}$ and \textit{a fortiori} regular in the sense of Milnor; we refer to \cite{Schmeding} for details.

For convenience we briefly comment on the outline of the paper. We start by reviewing the necessary background material on orbifolds, groupoids, pseudogroups, and the well-known construction of a groupoid from an orbifold and an orbifold atlas representing its orbifold structure. Groupoids which arise in this way will be called \textit{atlas groupoids}.  To overcome the problem that different orbifolds are identified with the same atlas groupoid we introduce, in Section~\ref{sec_marked}, a certain marking of atlas groupoids. It consists in attaching to an atlas groupoid a certain topological space and a certain homomorphism between its orbit space and the topological space. The general concept of marking already appeared in \cite{Moerdijk_survey}. There, however, the relation between a marking of a groupoid and an orbifold atlas (in local charts) is not discussed. The specific marking of an atlas groupoid introduced here allows us to recover the orbifold. There is a natural notion of homomorphisms between 
marked atlas groupoids. In Section~\ref{sec_homom} we 
characterize these homomorphisms in local charts. On the orbifold side, this characterization involves the choice of representatives of the orbifold structures, namely those orbifold atlases which were used to construct the marked atlas groupoids. Hence, at this point we get a notion of orbifold map with fixed representatives of orbifold structures, which we will call \textit{charted orbifold maps}. In Section~\ref{sec_redorbcat} we introduce a natural definition of composition of charted orbifold maps and a geometrically motivated definition of the identity morphism (a certain class of charted orbifold maps), which allows us to establish a natural equivalence relation on the class of charted orbifold maps. An orbifold map (which does not depend on the choice of orbifold atlases) is then an equivalence class of charted orbifold maps. The leading idea for this equivalence relation is geometric: we consider charted orbifold maps as equivalent if and only if they induce the same charted orbifold map on common 
refinements of the orbifold atlases. Moreover, using the same idea, we define the composition of orbifold maps. In this way, we construct a category of reduced orbifolds. Finally, in Section~\ref{atlascategory}, we characterize orbifolds as certain equivalence classes of marked atlas groupoids, and orbifold maps as equivalence classes of homomorphisms of marked atlas groupoids. These equivalence relations are natural adaptations of the classical Morita equivalence. In this way, there arises a category of marked atlas groupoids which is isomorphic to the orbifold category. As an additional benefit the isomorphism functor is constructive. As a final step, in Section~\ref{extension}, we show that the category of marked atlas groupoids is isomorphic to an analogously defined category of marked proper effective \'etale Lie groupoids (by embedding).

Even though the category of marked proper effective \'etale Lie groupoids constructed in this article not the classical one, all the classical groupoid (2- or bi-) categories can be recovered by considering elements in an equivalence class as stand-alone entities and the equivalence-providing maps as 2-morphisms, and by forgetting the marking. The ultimate orbifold category constructed here is the one which is satisfying from a geometric point of view (c.f.\@ the results in \cite{Schmeding}) and also allows to adapt without any difficulties all results on groupoids that are invariant under Morita equivalence (e.g.\@ the behavior of orbifold vector bundles under pullback).

We expect that all these constructions can be performed in a similar way for orbifolds over other fields, \eg complex ones, and for orbifolds with additional structures such as Riemannian ones.\\

\textit{Acknowledgments:}
This paper emerged from a workshop on orbifolds in 2007 which took place in Paderborn in the framework of the International Research Training Group 1133 ``Geometry and Analysis of Symmetries''. The author is very grateful to the participants of this workshop for their abiding interest and enlightening conversations. In particular, she likes to thank Joachim Hilgert, Alexander Schmeding and an anonymous referee for very careful reading of the manuscript. Moreover, she wishes to thank Dieter Mayer and the Institut f\"ur Theoretische Physik in Clausthal for the warm hospitality where part of this work was conducted. The author was partially supported by the International Research Training Group 1133 ``Geometry and Analysis of Symmetries'', the Sonderforschungsbereich/Transregio 45 ``Periods, moduli spaces and arithmetic of algebraic varieties'', the Max-Planck-Institut f\"ur Mathematik in Bonn, the SNF (200021-127145) and the Volkswagen Foundation.\\

\textit{Notation and conventions:} 
We use $\N_0 = \N  \cup \{0\}$ to denote the set of non-negative integers. If
not stated otherwise, every manifold is assumed to be real, paracompact, Hausdorff
and smooth ($C^\infty$). We also consider second-countable manifolds, and it will always be indicated whether we require a given manifold to be just paracompact or even second-countable. If $M$ is a manifold, then
$\Diff(M)$ denotes the group of diffeomorphisms of $M$. If $G$ is a subgroup of
$\Diff(M)$, then 
\[
 G\backslash M \sceq \{ Gm \mid m\in M\}
\]
denotes the space of $G$-orbits endowed with the final topology. If $A_1, A_2, B$ are sets (manifolds) and $f_1\co A_1\to B$, $f_2\co A_2\to B$ are maps (submersions), then we denote the fibered product of $f_1$ and $f_2$ by $A_1\pullback{f_1}{f_2}A_2$ and identify it with the set (manifold)
\[
 A_1\pullback{f_1}{f_2}A_2 = \{ (a_1,a_2) \in A_1\times A_2 \mid f_1(a_1) = f_2(a_2)\}.
\]
Finally, we say that a family $\mc V = \{ V_i \mid i\in I\}$ is \textit{indexed by $I$} if $I \to \mc V$, $i\mapsto V_i$, is a bijection.

\section{Reduced orbifolds, groupoids, and pseudogroups}\label{sec_redorb}

In this section we recall the necessary background on reduced orbifolds and groupoids. The notion of (reduced) orbifolds goes back to at least Satake \cite{Satake1, Satake2}, and has been refined since then, see \eg \cite{Haefliger_orbifold, Thurston, Moerdijk_Mrcun, Adem_Leida_Ruan}. Here, we use one of these more modern definitions of reduced orbifolds, and we consider different flavors for the manifolds involved. We recall that throughout any manifold is required to be real, smooth, Hausdorff and paracompact.  

\subsection{Reduced orbifolds}

Let $Q$ be a topological space, and $n\in \N_0$. A \textit{reduced orbifold chart} of
dimension $n$ on $Q$ is a triple $(V,G,\varphi)$ where $V$ is a
connected $n$--manifold without boundary, $G$ is a finite subgroup of $\Diff(V)$, and
$\varphi\co V\to Q$ is a map with open image $\varphi(V)$ that induces a
homeomorphism from $G\backslash V$ to $\varphi(V)$. In this case, $(V,G,\varphi)$ is said to
\textit{uniformize} $\varphi(V)$.  

One might argue that orbifold charts should better be called orbifold parametrizations. However, orbifold charts is the more established notion and we prefer to stick to it. The orbifold charts are called reduced (or effective) to indicate that the action of $G$ is effective.

Two reduced orbifold charts $(V,G,\varphi)$, $(W,H,\psi)$
on $Q$ are called \textit{compatible} if for each pair $(x,y)\in V\times W$ with
$\varphi(x) = \psi(y)$ there are open connected neighborhoods $\widetilde V$ of
$x$ and $\widetilde W$ of $y$ and a diffeomorphism $h\co \widetilde V\to
\widetilde W$ with $\psi\circ h = \varphi\vert_{\widetilde V}$. The map $h$
is called a \textit{change of charts}.  The neighborhoods $\wt V$ and $\wt W$ and the diffeomorphism $h$ can always be chosen in such a way that $h(x)=y$. Moreover $\widetilde V$ may assumed to be open
$G$-stable. This means that $\wt V$ is open and connected, and for each $g\in G$ we either have $g\wt V = \wt V$ or $g\wt V\cap \wt V = \emptyset$. In this case, $\widetilde W$ is open $H$-stable by \cite[Proposition~2.12(i)]{Moerdijk_Mrcun}.

A \textit{reduced orbifold atlas} of dimension $n$ on $Q$ is a collection of pairwise compatible reduced  orbifold charts 
\[ \mathcal V \sceq \{ (V_i, G_i, \varphi_i) \mid i\in I\} \]
of dimension $n$ on $Q$ such that $\bigcup_{i\in I} \varphi_i(V_i) = Q$. Two reduced  orbifold atlases are \textit{equivalent} if their union is a reduced orbifold atlas. A \textit{reduced orbifold structure} of dimension $n$ on $Q$ is an equivalence class of reduced orbifold atlases of dimension $n$ on $Q$. A \textit{reduced (paracompact resp.\@ second-countable) orbifold} of dimension $n$ is a pair $(Q,\mathcal U)$ where $Q$ is a (paracompact resp.\@ second-countable) Hausdorff space and $\mathcal U$ is a reduced  orbifold structure of dimension $n$ on $Q$. We note that since the manifolds $V$ in the orbifold charts $(V,G,\varphi)$ are assumed to be connected and finite-dimensional, paracompactness of $V$ is equivalent to second-countability of $V$. However, since the topological space $Q$ is not required to be connected, paracompactness does not necessarily imply second-countability of $Q$.

Let $\mc U$ be a reduced orbifold structure on $Q$. Each reduced orbifold atlas $\mc V$ in $\mc U$ is called a \textit{representative} of $\mc U$ or a \textit{reduced orbifold atlas of} $(Q,\mc U)$.

Since we are considering reduced orbifolds only, we omit the term ``reduced'' from now on. Moreover, when implicitly understood we will omit the terms ``paracompact'' and ``second-countable''.

Let $(V,G,\varphi)$, $(W,H,\psi)$ be orbifold charts of identical dimension on the topological space $Q$. Then an \textit{(open) embedding} $\mu\co (V,G,\varphi) \to (W,H,\psi)$ between these two orbifold charts is an (open) embedding $\mu\co V\to W$ between manifolds which satisfies $\psi\circ\mu = \varphi$. We remark that any embedding is automatically open by the Theorem on the Invariance of Domain.
If, in addition, $\mu$ is a diffeomorphism between $V$ and $W$, then $\mu$ is called an \textit{isomorphism} from $(V,G,\varphi)$ to $(W,H,\psi)$. Suppose that $S$ is an open $G$--stable subset of $V$ and let $G_S \sceq \{g\in G\mid gS=S\}$ denote the \textit{isotropy group} of $S$. Then $(S,G_S,\varphi\vert_S)$ is an orbifold chart on $Q$ as well, the \textit{restriction} of $(V,G,\varphi)$ to $S$.

\begin{remark}\label{groupiso}
Suppose that $\mu\co (V,G,\varphi) \to (W,H,\psi)$ is an embedding. By \cite[Proposition~2.12(i)]{Moerdijk_Mrcun},  $\mu(V)$ is an open $H$--stable subset of $W$,  and moreover that there is a unique group isomorphism $\overline\mu \co G \to H_{\mu(V)}$ for which $\mu(gx) = \overline\mu(g)\mu(x)$ for $g\in G$, $x\in V$.
\end{remark}

In the following example we provide two orbifolds with the same underlying
topological space. These orbifolds are particularly simple since both orbifold
structures have one-chart-representatives. Despite their simplicity they serve
as motivating examples for several definitions in this paper.

\begin{example}\label{notcompatible}
Let $Q\sceq [0,1)$ be endowed with the induced topology of $\R$. The map
\[
f\colon Q \to Q,\quad x  \mapsto  x^2
\]
is a homeomorphism. Further the map $\pr\co (-1,1) \to [0,1)$, $x\mapsto |x|$, induces a homeomorphism ${\{\pm\id \}}\backslash (-1,1) \to Q$. Then 
\[ V_1\sceq \big( (-1,1), \{\pm \id\}, \pr \big) \quad\text{and}\quad V_2\sceq \big( (-1,1), \{\pm\id\}, f\circ\pr\big) \]
are two orbifold charts on $Q$. They are easily seen to be non-compatible. 

Let $\mathcal U_1$ be the orbifold structure on $Q$ generated by $V_1$, and $\mathcal U_2$ be the one generated by $V_2$.
\end{example}

\subsection{Groupoids and homomorphisms}

A groupoid is 
a small category in which each morphism is an isomorphism. In the context of orbifolds this concept is most commonly expressed (equivalently) in terms of sets and maps.  The morphisms are then called arrows. 

A \textit{groupoid} $G$ is a
tuple $G=(G_0,G_1,s,t,m,u,i)$ consisting of the set $G_0$ of \textit{objects}, or the \textit{base} of $G$, the set $G_1$ of \textit{arrows}, and five \textit{structure maps}, namely the  \textit{source map} $s\co G_1\to G_0$, the \textit{target map} $t\co G_1\to G_0$,  the \textit{multiplication} or \textit{composition} $m\co G_1\pullback{s}{t} G_1\to G_1$, the \textit{unit map} $u\co G_0\to G_1$, and the \textit{inversion} $i\co G_1\to G_1$ which satisfy that
\begin{enumerate}[(i)]
\item\label{groupoidi} for all $(g,f)\in G_1\pullback{s}{t} G_1$ we have $s(m(g,f)) = s(f)$ and $t(m(g,f)) = t(g)$,
\item for all $(h,g),(g,f)\in G_1\pullback{s}{t} G_1$ we have $m( h, m(g,f)) = m(m(h,g),f)$,
\item for all $x\in G_0$ we have $s(u(x)) = x = t(u(x))$,
\item for all $x\in G_0$ and all $(u(x),f), (g,u(x))\in G_1\pullback{s}{t} G_1$
it follows $m(u(x),f) = f$ and $m(g,u(x)) = g$,
\item for all $g\in G_1$ we have $s(i(g)) = t(g)$ and $t(i(g)) = s(g)$,  and 
$m(g,i(g)) = u(t(g))$ and $m(i(g),g)= u(s(g))$.
\end{enumerate}
We often use the notations $m(g,f) = gf$, $u(x) = 1_x$, $i(g) = g^{-1}$, and
$g\co x\to y$ or $\stackrel{g}{x\to y}$ for an arrow $g\in G_1$ with
$s(g)=x$, $t(g)=y$. Moreover, $G(x,y)$ denotes the set of arrows from $x$ to
$y$.

Classically, the base space of a Lie groupoid is required to be a second-countable manifold. However, parts of the theory of Lie groupoids stay valid if the base manifold is just paracompact. Therefore, as with orbifolds, we consider two flavors of Lie groupoids here.

A \textit{(paracompact resp.\@ second-countable) Lie groupoid} is a groupoid $G$ for which $G_0$ is a 
paracompact resp.\@ second-countable manifold, $G_1$ is a smooth (possibly non-Hausdorff, possible non-second-countable)
manifold, the structure maps $s,t\co G_1\to G_0$ are smooth
submersions (hence $G_1\pullback{s}{t} G_1$, the domain of $m$, is a
smooth, possibly non-Haus\-dorff manifold),  and the structure maps $m,u$ and $i$ are smooth. 
A Lie groupoid is called \textit{\'etale} if its source and target map are local diffeomorphisms. It is called \textit{proper} if $G_1$ is Hausdorff and the map $(s,t)\colon G_1 \to G_0\times G_0$ is proper. An \'etale Lie groupoid  $G$ is called \textit{effective} if the map 
\[
 g \mapsto \germ_{s(g)}(t\circ (s\vert_U)^{-1}),
\]
where $g\in G_1$ and $U$ is an open neighborhood of $g$ in $G_1$ such that $t\vert_U$ and $s\vert_U$ are injective, is injective.

Let $G$ and $H$ be groupoids.
A \textit{homomorphism} from $G$ to $H$ is a functor $\varphi\co G\to H$,
\ie it is a tuple $\varphi=(\varphi_0,\varphi_1)$ of maps $\varphi_0\co G_0\to H_0$ and $\varphi_1\co G_1\to H_1$
which commute with all structure maps. If $G$ and $H$ are Lie groupoids, then $\varphi$ is a homomorphism between them if it is a homomorphism of the abstract groupoids with the additional requirement that  $\varphi_0$ and $\varphi_1$ be smooth maps.

Let $G$ be a
groupoid. The \textit{orbit} of $x\in G_0$ is the set 
\[ Gx \sceq t(s^{-1}(x)) = \left\{ y\in G_0\left\vert\ \exists\, g\in G_1\co
x\stackrel{g}{\rightarrow} y\right.\right\}.\]
Two elements $x,y\in G_0$ are called \textit{equivalent}, $x\sim y$, if
they are in the same orbit. The quotient space $|G|\sceq G_0/_\sim$ is called the
\textit{orbit space} of $G$.  The canonical quotient map
$G_0 \to |G|$ is denoted by $\pr$ or $\pr_G$, and  $[x] \sceq \pr(x)$ for $x\in G_0$.

\subsection{Pseudogroups and groupoids}

We recall how to construct a Lie groupoid from an orbifold and a representative of its orbifold structure. This construction is well known, see e.g.\@ the book by Moerdjik and Mr\v{c}un \cite{Moerdijk_Mrcun}. We provide it here for the convenience of the reader and to introduce the notations we will use later on. It is a two-step process in which one first assigns a pseudogroup to the orbifold, which depends on the representative of the orbifold structure. Then one constructs an \'{e}tale Lie groupoid from the pseudogroup. For reasons of generality and clarity we start with the second step.

\begin{defi}\label{def_pseudogroup}
Let $M$ be a manifold. A \textit{transition} on $M$ is a
diffeomorphism $f\co U\to V$ where $U,V$ are open subsets of $M$. In particular, the empty map
$\emptyset\to\emptyset$ is a transition on $M$. The \textit{product} of two
transitions $f\co U\to V$, $g\co U'\to V'$ is the transition
\[ f\circ g\co g^{-1}(U\cap V') \to f(U\cap V'),\ x\mapsto f(g(x)).\]
The \textit{inverse} of $f$ is the inverse of $f$ as a function. If $f\co U\to V$ is a transition, we use $\dom f$ to denote its
\textit{domain} and $\cod f$ to denote its \textit{codomain}. Further, if $x\in \dom f$, then $\germ_x f$ denotes the
germ of $f$ at $x$.

Let $\mathcal A(M)$ be the set of all transitions on $M$. A \textit{pseudogroup}
on $M$ is a subset $P$ of $\mathcal A(M)$ which is closed under multiplication
and inversion. 
A pseudogroup $P$ is called \textit{full} if $\id_U\in P$ for each open subset
$U$ of $M$. It is said to be \textit{complete} if it is full and satisfies
the following  gluing property: Whenever there is a transition $f\in \mathcal
A(M)$ and an open covering $(U_i)_{i\in I}$ of $\dom f$ such that
$f\vert_{U_i}\in P$ for all $i\in I$, then $f\in P$.
\end{defi}

We now recall how to construct an \'etale Lie groupoid from a full pseudogroup.

\begin{construction}\label{constr_groupoid} Let $M$ be a manifold and $P$ a full pseudo\-gr\-oup on $M$. 
The \textit{associated groupoid} $\Gamma\sceq \Gamma(P)$ is given by
\[ \Gamma_0 \sceq M,\quad \Gamma_1\sceq \{ \germ_x f\mid f\in P,\ x\in\dom f\},\]
and
\[ \Gamma(x,y) \sceq \{ \germ_x f \mid f\in P,\ x\in\dom f,\ f(x) = y\}.\]
For $f\in P$ define $U_f \sceq \left\{ \germ_x f\left\vert\ x\in\dom f \vphantom{\germ_x f}\right.\right\}$.
The topology and differential structure of $\Gamma_1$ is given by the germ topology and germ differential structure, that is, for each $f\in P$ the bijection 
\[ \varphi_f\co
\left\{
\begin{array}{ccc}
U_f  & \to & \dom f
\\
\germ_x f & \mapsto & x
\end{array}
\right.
\]
is required to be a diffeomorphism. The structure maps $(s,t,m,u,i)$ of $\Gamma$ are the obvious ones, namely
\begin{align*}
s(\germ_x f) & \sceq x
\\
t(\germ_x f) & \sceq f(x)
\\
m(\germ_{f(x)} g, \germ_x f) & \sceq \germ_x(g\circ f)
\\
u(x) & \sceq \germ_x\id_U 
\text{ for an open neighborhood $U$ of $x$}
\\
i(\germ_x f) & \sceq \germ_{f(x)} f^{-1}.
\end{align*}
Obviously, $\Gamma(P)$ is an \'etale Lie groupoid.
\end{construction}

\begin{specialcase}\label{atlasgroupoid}
Let $(Q,\mathcal U)$ be an orbifold, and let 
\[
\mathcal V = \{ (V_i, G_i, \pi_i) \mid i\in I \}
\]
be a representative of $\mc U$ indexed by $I$. 
We define
\[ V\sceq \coprod_{i\in I} V_i\quad\text{and}\quad \pi\sceq \coprod_{i\in I}\pi_i.\]
Then
\[ \Psi(\mathcal V) \sceq \big\{ \text{$f$ transition on $V$} \ \big\vert\  \pi\circ f = \pi\vert_{\dom f} \big\}.\]
is a complete pseudogroup on $V$. The associated groupoid 
\[
\Gamma(\mathcal V)\sceq \Gamma(\Psi(\mathcal V))
\]
is the proper effective \'etale Lie groupoid we shall associate to $Q$ and $\mathcal V$. Note that this groupoid depends on the choice of the representative of the orbifold structure $\mc U$ of $Q$. A groupoid which arises in this way we call \textit{atlas groupoid}.

If $(Q,\mc U)$ is a second-countable orbifold, then we can choose $\mc V$ to be countable and the associated atlas groupoid has a second-countable base. Whenever dealing with second-countable orbifolds in the following, we will assume that the chosen representative of the orbifold structure is countable.
\end{specialcase}

\begin{example}\label{orbexample}
Recall the orbifolds $(Q,\mathcal U_i)$ $(i=1,2)$ from Example~\ref{notcompatible}, and consider the representative $\mathcal V_i \sceq \{V_i\}$ of $\mathcal U_i$.  Proposition~2.12 in \cite{Moerdijk_Mrcun} implies that 
\[ \Psi(\mathcal V_i) = \big\{ g\vert_U\co U \to g(U) \ \big\vert\ \text{$U\subseteq (-1,1)$ open, $g\in \{\pm \id\}$}   \big\}. \]
In both cases the associated groupoid $\Gamma\sceq \Gamma(\mathcal V_i)$ is 
\begin{align*}
\Gamma_0 & = (-1,1)
\\
\Gamma(x,y) & =
\begin{cases}
\big\{ \germ_0 \id, \germ_0 (-\id) \big\} & \text{if $x=0=y$}
\\
\big\{ \germ_x \id \big\} & \text{if $x=y\not=0$}
\\
\big\{ \germ_x (-\id) \big\} & \text{if $x=-y\not=0$}
\\
\emptyset & \text{otherwise.}
\end{cases}
\end{align*}

\end{example}

\section{Marked Lie groupoids and their homomorphisms}\label{sec_marked}

In Example~\ref{orbexample} we have seen that it may happen that the same atlas
groupoid is associated to two different orbifolds. The reason for this is that
in the definition of the pseudogroup which is needed for the construction of the atlas groupoid one loses information about the projection
maps $\varphi$ of the orbifold charts $(V,G,\varphi)$. To be able to distinguish atlas
groupoids constructed from different orbifolds, we mark the groupoids with a
topological space and a homeomorphism. It will turn out that this marking
suffices to identify the orbifold one started with.

A \textit{marked Lie groupoid} is a triple $(G, \alpha , X)$ consisting of a
Lie groupoid $G$, a topological space $X$, and a homeomorphism $\alpha\co|G|
\to X$. 

For atlas groupoids there exists a specific marking which is crucial for the isomorphism between the orbifold category and the groupoid category. It is stated in the following lemma of which we omit its straightforward proof.

\begin{lemma}\label{atlashomeom}
Let $(Q,\mathcal U)$ be an orbifold and 
\[
\mathcal V = \{ (V_i,G_i,\pi_i) \mid i\in I\}
\]
a representative of $\mc U$ indexed by $I$. Set $V\sceq \coprod_{i\in I} V_i$ and $\pi\sceq \coprod_{i\in I}
\pi_i \co V\to Q$. Then the map
\[ 
\alpha \co 
\left\{ 
\begin{array}{ccc}
|\Gamma(\mathcal V)| & \to & Q
\\{}
[x] & \mapsto & \pi(x)
\end{array}
\right.
\]
is a homeomorphism.
\end{lemma}

Let $(Q,\mathcal U)$ be an orbifold. To each (countable if $(Q,\mc U)$ is second-countable) orbifold atlas $\mathcal V$ of $Q$
we assign the marked atlas groupoid $(\Gamma(\mathcal V), \alpha_{\mathcal V}, Q)$
 with $\alpha_{\mathcal V}$ being the homeomorphism from
Lemma~\ref{atlashomeom}. We often only write $\Gamma(\mathcal V)$ to refer
to this marked groupoid.

\begin{example}
Recall from Example~\ref{orbexample} the orbifolds $(Q,\mathcal U_i)$ for $i=1,2$, their respective orbifold atlases $\mathcal V_i$, and the associated grou\-poids $\Gamma=\Gamma(\mathcal V_i)$. The orbit of $x\in \Gamma_0$ is 
$\{x,-x\}$. Hence the homeomorphism associated to $(Q,\mc U_i)$ is $\alpha_{\mathcal V_i}\co |\Gamma| \to Q$ given by $\alpha_{\mathcal V_1}([x]) =  |x|$ resp.\@ $\alpha_{\mathcal V_2}([x]) =   x^2$. Thus, the associated marked groupoids $(\Gamma, \alpha_{\mc V_1},Q)$ and
$(\Gamma,\alpha_{\mc V_2},Q)$ are different.
\end{example}

\begin{prop}\label{backorb}
Let $(Q,\mc U)$ and $(Q',\mc U')$ be orbifolds. Suppose that $\mc V$ is a
representative of $\mc U$, and $\mc V'$ a representative of $\mc U'$. If the
associated marked atlas groupoids $(\Gamma(\mc V), \alpha_{\mc V}, Q)$ and
$(\Gamma(\mc V'), \alpha_{\mc V'}, Q')$ are equal, then the orbifolds $(Q,\mc
U)$ and $(Q',\mc U')$ are equal. More precisely, we even have $\mc V = \mc V'$.
\end{prop}

\begin{proof}
Clearly, $Q=Q'$. Suppose
that 
\[
 \mc V = \{ (V_i, G_i, \pi_i) \mid i\in I\} \quad\text{and}\quad \mc V' = \{
(V'_j, G'_j, \pi'_j) \mid j\in J\},
\]
indexed by $I$ resp.\@ by $J$. From $\Gamma(\mc V) = \Gamma(\mc V')$  it follows that 
\[
  \coprod_{i\in I} V_i = \Gamma(\mc V)_0 = \Gamma(\mc V')_0 = \coprod_{j\in J}
V'_j.
\]
Since each $V_i$ and each $V'_j$ is connected, there is a bijection between $I$ and $J$.
We may assume $I=J$ and $V_i = V'_i$
for all $i\in I$. Let $x\in V_i$. Then $\pi_i(x) = \alpha_{\mc V}([x]) = \alpha_{\mc V'}([x]) =
\pi'_i(x)$. Therefore $\pi_i=\pi'_i$ for all $i\in I$. Thus $\Psi(\mc V) = \Psi(\mc V')$. Moreover
\[
 G_i = \{ f\in \Psi(\mc V) \mid \dom f= V_i = \cod f\} = G'_i.
\]
To show that the actions $G_i$ and $G'_i$ on $V_i$ are equal, let $g\in G_i$. For each $x\in V_i$ we have $\pi'_i( g(x)) =
\pi'_i(x)$. This shows that $g(x) \in G'_ix$ for each $x\in V_i$. By
\cite[Lemma~2.11]{Moerdijk_Mrcun} there exists a unique element $g'\in G'_i$
such that $g=g'$. From this it follows that $G_i = G'_i$ as acting groups. Thus, $\mc V  = \mc V'$.
\end{proof}

\begin{lemma}
 Let $(G,\alpha, X)$ and $(H,\beta, Y)$ be marked Lie groupoids and suppose that $\varphi = (\varphi_0,\varphi_1)\co G\to H$  is a homomorphism of Lie groupoids. Then $\varphi$ induces a unique map $\psi$ such that the diagram
\[
\xymatrix{
G_0  \ar[r]^{\pr_G} \ar[d]_{\varphi_0}& |G| \ar[r]^\alpha & X\ar[d]^\psi
\\
H_0 \ar[r]^{\pr_H} & |H| \ar[r]^\beta & Y
}
\]
commutes. Moreover, $\psi$ is continuous.
\end{lemma}

\begin{proof}
The map $\varphi$ induces a unique map $|\varphi|\co |G|\to |H|$ such that  $|\varphi|\circ \pr_G = \pr_H\circ \varphi_0$, which is continuous. Then $\psi =\beta\circ |\varphi| \circ \alpha^{-1}$.
\end{proof}

\section{Groupoid homomorphisms in local charts}\label{sec_homom}

In this section we characterize homomorphisms between marked atlas groupoids on the orbifold side, \ie in terms of local charts. We proceed in a two-step process. First we define a concept which we call representatives of orbifold maps. Each representative of an orbifold map gives rise to exactly one homomorphism between the associated marked atlas groupoids. Since, in general, each groupoid homomorphism corresponds to several such representatives, we then impose an equivalence relation on the class of all representatives for fixed orbifold atlases. The equivalence classes, called charted orbifold maps, turn out to be in bijection with the homomorphisms between the marked atlas groupoids. The constructions in this section are subject to a fixed choice of representatives of the orbifold structures. In the following sections we will use this construction as a basic building block for a notion of maps (or morphisms) between orbifolds which is independent of the chosen representatives.

Throughout this section let $(Q,\mathcal U)$, $(Q',\mathcal U')$ denote two orbifolds.

\begin{defi} Let $f\co Q\to Q'$ be a continuous map, and suppose that
$(V,G,\pi)\in\mathcal U$, $(V',G',\pi')\in\mathcal U'$ are  orbifold charts. A
\textit{local lift} of $f$ with respect to $(V,G,\pi)$ and $(V',G',\pi')$ is a smooth map 
$\tilde f\co V\to V'$ such that $\pi'\circ\tilde f = f\circ\pi$.
In this case, we call $\tilde f$ a \textit{local lift of $f$ at $q$} for each $q\in \pi(V)$.
\end{defi}

Recall the pseudogroup $\mathcal A(M)$ from Definition~\ref{def_pseudogroup}.

\begin{defi}\label{generatepsgr}
Let $M$ be a manifold and $A$ a pseudogroup on $M$ which satisfies the gluing property from Definition~\ref{def_pseudogroup} and is closed under restrictions. The latter means that if $f\in A$ and $U\subseteq \dom f$ is open, then the map $f\vert_U\co U \to f(U)$ is in $A$. Suppose that $B$ is a subset of $\mathcal A(M)$. Then $A$ is said to be \textit{generated} by $B$ if $B\subseteq A$ and for each $f\in A$ and each $x\in\dom f$ there exists some $g\in B$ with $x\in \dom g$ and an open set $U\subseteq \dom f \cap \dom g$ such that $x\in U$ and $f\vert_U = g\vert_U$.
In this case we say that $B$ \textit{generates} $A$.
\end{defi}

\begin{defi} Let $M$ be a manifold. A subset $P$ of $\mathcal A(M)$ is called a \textit{quasi-pseudogroup} on $M$ if it satisfies the following two properties:
\begin{enumerate}[(i)]
\item If $f\in P$ and $x\in \dom f$, then there exists an open set $U$ with $x\in U\subseteq \dom f$ and $g\in P$ such that there exists an open set $V$ with $f(x)\in V\subseteq \dom g$ and 
\[
 \big( f\vert_U\big)^{-1} = g\vert_V.
\]
\item If $f,g\in P$ and $x\in f^{-1}(\cod f \cap \dom g)$, then there exists $h\in P$ with $x\in \dom h$ such that we find an open set $U$ with 
\[
x\in U \subseteq f^{-1}(\cod f \cap \dom g) \cap \dom h\quad\text{and}\quad
 g\circ f\vert_U = h\vert_U.
\]
\end{enumerate}
\end{defi}

A quasi-pseudogroup is designed to work with the germs of its elements. Therefore identities (like inversion and composition) of elements in quasi-pseudogroups are only required to be satisfied locally, whereas for (ordinary) pseudogroups these identities have to be valid globally. One easily proves that each quasi-pseudogroup generates a unique pseudogroup which satisfies the gluing property from Definition~\ref{def_pseudogroup} and is closed under restrictions. Conversely, each generating set for such a pseudogroup is necessarily  a quasi-pseudogroup.

In the following definition of a representative of an orbifold map, the
underlying continuous map $f$ is the only entity which is stable under change of orbifold
atlases or, in other words, under the choice of local lifts. The pair $(P,\nu)$
should be considered as one entity. It serves as a transport of changes of
charts from one orbifold to another. Here we ask for a quasi-pseudogroup $P$ instead of working with all of
$\Psi(\mathcal V)$ (recall $\Psi(\mc V)$ from Special Case~\ref{atlasgroupoid}) for two reasons. In general, $P$ is much
smaller than $\Psi(\mathcal V)$. Sometimes it may even be finite. In 
Example~\ref{Pfinite} below we see that for some orbifolds, $P$ may
consist of only two elements. Moreover, if the orbifold is a connected manifold, $P$ can always be
chosen to be $\{\id\}$. The other reason is that it is much easier
to construct a quasi-pseudogroup $P$ and a compatible map $\nu$ from a given groupoid
homomorphism than a map $\nu$ defined on all of
$\Psi(\mathcal V)$.

Examples~\ref{notf} and \ref{Pfinite} below show that the objects requested in the following definition need not exist nor, if they exist, are uniquely determined. 

\begin{defi}\label{repchartedorbmap} A \textit{representative of an orbifold map} from $(Q,\mathcal U)$ to $(Q',\mathcal U')$ is a tuple 
\[ \hat f\sceq (f, \{\tilde f_i\}_{i\in I}, P, \nu )\]
where 
\begin{enumerate}[(R1)]
\item \label{f} $f\co Q\to Q'$ is a continuous map,
\item \label{liftings} for each $i\in I$, the map $\tilde f_i$ is a local lift of $f$ \wrt some orbifold charts $(V_i, G_i,\pi_i)\in\mathcal U$, $(V'_i, G'_i, \pi'_i)\in\mathcal U'$ such that
\[ \bigcup_{i\in I}\pi_i(V_i) = Q \]
and $(V_i, G_i, \pi_i) \not= (V_j, G_j, \pi_j)$ for $i,j\in I$, $i\not= j$,
\item \label{P} $P$ is a quasi-pseudogroup which consists of changes of charts
of the orbifold atlas
\[ \mathcal V \sceq \{ (V_i, G_i, \pi_i)\mid i\in I \}\]
of $(Q,\mc U)$ and generates $\Psi(\mathcal V)$.
\item \label{nu} Let $\psi\sceq \coprod_{i\in I} \tilde f_i$ and let $\mc V'$ be a representative of $\mc U'$ which contains
\[
 \{ (V'_i, G'_i, \pi'_i) \mid i\in I\}.
\]
Then $\nu\co P \to \Psi(\mc V')$ is a map which assigns
to each $\lambda\in P$ an embedding 
\[ \nu(\lambda) \co (W', H', \chi') \to (V', G', \varphi')\]
between some orbifold charts in $\mathcal V'$ such that
\begin{enumerate}[(a)]
\item \label{invariant} $\psi\circ\lambda = \nu(\lambda)\circ \psi\vert_{\dom \lambda}$,
\item \label{welldefined} for all $\lambda,\mu\in P$ and all $x\in\dom\lambda\cap\dom\mu$ with $\germ_x\lambda=\germ_x\mu$, we have 
\[
\germ_{\psi(x)}\nu(\lambda) = \germ_{\psi(x)}\nu(\mu),
\]
\item \label{mult_pg} for all $\lambda,\mu\in P$, for all $x\in \lambda^{-1}(\cod\lambda\cap\dom\mu)$ we have 
\[ \germ_{\psi(\lambda(x))}\nu(\mu) \cdot \germ_{\psi(x)}\nu(\lambda) = \germ_{\psi(x)}\nu(h),\]
where $h$ is an element of $P$ with $x\in\dom h$ such that there is an open set $U$ with 
\[
x\in U\subseteq \lambda^{-1}(\cod\lambda\cap\dom \mu)\cap\dom h
\]
and $\mu\circ\lambda\vert_U = h\vert_U$,
\item \label{unit_pg} for all $\lambda\in P$ and all $x\in \dom \lambda$ such that there exists an open set $U$ with $x\in U\subseteq\dom \lambda$ and $\lambda\vert_U = \id_U$ we have 
\[
\germ_{\psi(x)}\nu(\lambda) = \germ_{\psi(x)}\id_{U'}
\]
with $U'\sceq \coprod_{i\in I} V'_i$.
\end{enumerate}
\end{enumerate}
The orbifold atlas $\mathcal V$ is called the \textit{domain atlas} of the
representative $\hat f$, and the set
\[ \{ (V'_i, G'_i, \pi'_i) \mid i\in I\}\]
is called the \textit{range family} of $\hat f$. The latter set is not necessarily indexed by $I$.
\end{defi}

Condition (\apref{R}{mult_pg}{}) is in fact independent of the choice of $h$. The technical (and easily satisfied) condition in (\apref{R}{liftings}{}) that each two orbifold charts in $\mc V$ be distinct is required because we use $I$ as an index set for $\mathcal V$ in (\apref{R}{P}{}) and other places.

Example~\ref{notf} below shows that the continuous map $f$ in (\apref{R}{f}{})
cannot be chosen arbitrarily. It is not even sufficient for $f$ to be any
homeomorphism.

\begin{example}\label{notf}
Recall the orbifold $(Q,\mathcal U_1)$ from Example~\ref{notcompatible}. The map 
\[
f\co Q\to Q, \quad f(x) = \sqrt{x},
\]
is a homeomorphism on $Q$. We show that $f$ has no local lift at $0$. Each orbifold chart in $\mc U_1$ that uniformizes a neighborhood of $0$ is isomorphic to an orbifold chart of the form $(I, \{\pm \id_I\}, \pr)$  where $I=(-a,a)$ for some $0<a<1$. Seeking a contradiction assume that $\tilde f$ is a local lift of $f$ at $0$ with domain $I = (-a,a)$. For each $x\in I$, necessarily $\tilde f(x) \in \big\{ \pm \sqrt{|x|} \big\}$. Since $\tilde f$ is required to be continuous, there remain four possible candidates for $\tilde f$, namely
\begin{align*}
\tilde f_1(x) &= \sqrt{ |x| },& \tilde f_2 &= -\tilde f_1,
\\
\tilde f_3 (x) & = 
\begin{cases}
\sqrt{x} & x\geq 0
\\
-\sqrt{-x} & x\leq 0,
\end{cases}
&
\tilde f_4 & = -\tilde f_3.
\end{align*}
But none of these is differentiable in $x=0$, hence there is no local lift of $f$ at $0$.
\end{example}

The following example shows that the pair $(P,\nu)$ is not uniquely determined
by the choice of the family of local lifts.

\begin{example}\label{Pfinite} Recall the orbifold $(Q,\mathcal U_1)$  and the representative $\mathcal V_1 = \{V_1\}$ of $\mathcal U_1$ from Example~\ref{notcompatible}.  The map $f\co Q\to Q$, $q\mapsto 0$, is clearly continuous and has the local lift 
\[ 
\tilde f \co\left\{
\begin{array}{ccc}
(-1,1) & \to & (-1,1)
\\
x & \mapsto & 0
\end{array}\right.
\]
with respect to $V_1$ and $V_1$. Consider the quasi-pseudogroup $P=\{ \pm\id_{(-1,1)} \}$
on $(-1,1)$. Proposition~2.12 in \cite{Moerdijk_Mrcun} implies that $P$
generates $\Psi(\mathcal V_1)$. The triple $(f,\tilde f, P)$ can be completed in
the following two different ways to representatives of orbifold maps on $(Q,\mc U_1)$:
\begin{enumerate}[(a)]
\item $\nu_1(\pm \id_{(-1,1)}) \sceq \id_{(-1,1)}$,
\item $\nu_2(\id_{(-1,1)}) \sceq \id_{(-1,1)}$, $\nu_2(-\id_{(-1,1)}) \sceq -\id_{(-1,1)}$.
\end{enumerate}
We will see in Example~\ref{differenthomoms} below that $(f,\tilde f, P,\nu_1)$ and $(f,\tilde f, P,\nu_2)$ give rise to different groupoid homomorphisms.
\end{example}

\begin{prop}\label{orb_gr} Let $\hat f = (f, \{\tilde f_i\}_{i\in I}, P, \nu )$
be a representative of an orbifold map from $(Q,\mathcal U)$ to $(Q',\mathcal U')$.
Suppose that $\mathcal V= \{ (V_i,G_i,\pi_i) \mid i\in I\}$, is the domain atlas of $\hat f$, which is an orbifold atlas of $(Q,\mc U)$ indexed by $I$. Let $\mathcal V'$ be an orbifold atlas of $(Q',\mc U')$ which contains the range family $\big\{ (V'_i, G'_i, \pi'_i) \ \big\vert\ i\in I  \big\}$. 
Define the map $\varphi_0\co \Gamma(\mathcal V)_0 \to \Gamma(\mathcal V')_0$
by  
\[ \varphi_0 \sceq \coprod_{i\in I}\tilde f_i.\]
Suppose that $\varphi_1\co \Gamma(\mathcal V)_1 \to \Gamma(\mathcal V')_1$ is
determined by
\[ \varphi_1(\germ_x\lambda) \sceq \germ_{\varphi_0(x)}\nu(\lambda)\]
for all $\lambda\in P$, $x\in\dom\lambda$.
Then 
\[ \varphi = (\varphi_0,\varphi_1) \co \Gamma(\mathcal V) \to \Gamma(\mathcal V') \]
is a homomorphism.
Moreover, $\alpha_{\mathcal V'}\circ |\varphi| = f\circ \alpha_{\mathcal
V}$.
\end{prop}

\begin{proof} 
Obviously, $\varphi_0$ is smooth. To show that $\varphi_1$ is a well-defined map on all of $\Gamma(\mc V)_1$, let $g\in\Psi(\mathcal V)$ and $x\in \dom g$. Then there exists $\lambda\in P$ such that $x\in\dom\lambda$ and
\[ g\vert_U = \lambda\vert_U\]
for some open subset $U\subseteq \dom g \cap \dom \lambda$ with $x\in U$. Hence $\germ_x g = \germ_x\lambda$. Thus 
\[ \varphi_1(\germ_x g) = \varphi_1(\germ_x\lambda) = \germ_{\varphi_0(x)}\nu(\lambda).\]
If there is $\mu\in P$ such that $x\in\dom\mu$ and $g\vert_W =
\mu\vert_W$ for some open subset $W$ of $\dom g \cap \dom \mu$ with $x\in W$, then $\germ_x \mu= \germ_x\lambda$. By
(\apref{R}{welldefined}{}), $\germ_{\varphi_0(x)}\nu(\mu) = \germ_{\varphi_0(x)}\nu(\lambda)$ and thus
\[ \varphi_1(\germ_x \mu) = \varphi_1(\germ_x\lambda).\]
This shows that $\varphi_1$ is indeed well-defined on all of
$\Gamma(\mathcal V)_1$. The properties (\apref{R}{invariant}{}), (\apref{R}{mult_pg}{}) and
(\apref{R}{unit_pg}{}) yield that $\varphi$ commutes with the  structure maps. It remains to show that $\varphi_1$ is smooth. For this, let $\germ_x\lambda\in \Gamma(\mathcal V)_1$ with $\lambda\in P$. The definition of $\nu$ shows that $\varphi_1$ maps
\[ U\sceq \{ \germ_y\lambda\mid y\in \dom\lambda\}\]
to
\[ U'\sceq \{ \germ_z\nu(\lambda) \mid z\in\dom\nu(\lambda) \}.\]
Now the diagram
\[
\xymatrix{
U \ar[d]_s\ar[r]^{\varphi_1} & U'\ar[d]^s  &&   \germ_y\lambda
\ar@{|->}[r] \ar@{|->}[d] & \germ_{\varphi_0(y)}\nu(\lambda)\ar@{|->}[d]
\\
\dom\lambda \ar[r]^{\varphi_0} & \dom\nu(\lambda) &&  y \ar@{|->}[r] & \varphi_0(y)
}
\]
commutes, the vertical maps (restriction of source maps) are diffeomorphisms and $\varphi_0$ is smooth, so $\varphi_1$ is smooth.
Finally, suppose $x\in V_i$. Then 
\begin{align*}
 \left(\alpha_{\mathcal V'}\circ|\varphi|\right)\big([x]\big) & = \alpha_{\mathcal V'}\big([\varphi_0(x)]\big) =  \pi'_i(\tilde f_i(x)) = f(\pi_i(x))
= \left(f\circ \alpha_{\mathcal V}\right)\big([x]\big).
\end{align*}
This completes the proof.
\end{proof}

Example~\ref{differenthomoms} below shows that the choice of $(P,\nu)$ in Definition~\ref{repchartedorbmap} is not unique, and more elaborated examples would immediately show that not each pair $(P,\nu)$ automatically satisfies the compatibility conditions in \apref{R}{invariant}{}). However, for some important types of maps between orbifolds (e.g.\@ isomorphisms) the choice of $(P,\nu)$ is canonical and the compatibility conditions are also automatic, which is the essence of Proposition~\ref{extending} below.

\begin{example}\label{differenthomoms}
Recall the setting of Example~\ref{Pfinite} and the associated groupoid $\Gamma\sceq\Gamma(\mathcal V_1)$ from Example~\ref{orbexample}. The homomorphism $\varphi=(\varphi_0,\varphi_1)$ of $\Gamma$ induced by $(f,\tilde f, P, \nu_1)$ is $\varphi_0 = \tilde f$ and
\[ \varphi_1( \germ_x(\pm \id_{(-1,1)}) ) = \germ_0 \id_{(-1,1)}.\]
The homomorphism $\psi=(\psi_0,\psi_1)\co\Gamma\to \Gamma$ induced by $(f,\tilde f, P,\nu_2)$ is $\psi_0=\tilde f$ and 
\begin{align*}
 \psi_1(\germ_x\id_{(-1,1)}) & = \germ_0 \id_{(-1,1)},
\\
 \psi_1(\germ_x(-\id_{(-1,1)}) ) & = \germ_0(-\id_{(-1,1)}).
\end{align*}
\end{example}

The following proposition is the converse to Proposition~\ref{orb_gr}. Its proof is constructive. In Section~\ref{atlascategory} we will use this construction to define the functor between the category of orbifolds and that of marked atlas groupoids.

\begin{prop}\label{inducedorbgr}
Let $\mathcal V$ be a representative of $\mc U$, $\mathcal V'$ a representative of $\mc U'$, and
\[ \varphi = (\varphi_0,\varphi_1)\co \Gamma(\mathcal V) \to \Gamma(\mathcal V') \]
a homomorphism.
Then $\varphi$ induces a representative of an orbifold map 
\[ (f, \{\tilde f_i\}_{i\in I}, P, \nu) \]
with domain atlas $\mathcal V$, range family contained in $\mathcal V'$, and
\[ \tilde f_i = \varphi_0\vert_{V_i}\]
for all $i\in I$. Moreover, we have $f= \alpha_{\mathcal V'}\circ|\varphi|\circ \alpha_{\mathcal V}^{-1}$.
\end{prop}

\begin{proof} 
We start by showing that for each $h\in\Psi(\mathcal V)$ and each $x\in\dom h$ there exist an element $g\in\Psi(\mathcal V')$ and an open  neighborhood $U$ of $x$ (which may depend on $g$) with $U\subseteq \dom h$ such that for each $y\in U$ we have
\begin{equation}\label{thickidentity}
\varphi_1(\germ_y h) = \germ_{\varphi_0(y)} g.
\end{equation} 
So, let $h\in\Psi(\mc V)$ and $x\in \dom h$. By the definitions of $\Gamma(\mathcal V)_1$ and $\varphi_1$, there exists $g\in\Psi(\mathcal V')$ such that
\[ \varphi_1(\germ_x h) = \germ_{\varphi_0(x)} g.\]
Since $\varphi_1$ is continuous, the preimage of the $\germ_{\varphi_0(x)} g$--neighborhood
\[ U'_g = \{ \germ_z g\mid z\in\dom g \} \]
is a neighborhood of $\germ_x h$. Hence there exists an open neighborhood $U$ of $x$ with $U\subseteq \dom h$ such that 
\[ U_{h\vert U} = \{ \germ_y h\mid y\in U\} \subseteq \varphi_1^{-1}\left(U'_g\right).\]
Thus, for all $y\in U$ we have \eqref{thickidentity} as claimed. We remark that each two possible choices for $g$ coincide on some neighborhood of $\varphi_0(x)$.

For each $h\in\Psi(\mc V)$ and each $x\in\dom h$ we now choose a pair $(g,U)$ where $g\in\Psi(\mc V')$ is an embedding between some orbifold charts in $\mc U'$ and $U$ is an open neighborhood of $x$ such that $h\vert_U$ is a change of charts of $\mc V$. Let $P(h,x)\sceq (g,U)$. We adjust choices such that for $h_1,h_2\in\Psi(\mc V)$ and $x_1\in\dom h_1$, $x_2\in\dom h_2$ the chosen pairs $P(h_1,x_1) = (g_1,U_1)$ resp.\@ $P(h_2,x_2) = (g_2,U_2)$ are either equal or $U_1\not=U_2$. Let $P$ denote the family of the changes of charts we have chosen in this way:
\[
 P  = \{ h\vert_U\co U \to h(U)\mid h\in\Psi(\mc V),\ x\in\dom h,\ P(h,x) = (g,U)\}.
\]
By construction, $P$ is a quasi-pseudogroup which generates
$\Psi(\mathcal V)$. We define the map $\nu\co P \to \Psi(\mc V')$ by 
\[ \nu(\lambda) \sceq g\]
where $g$ is the unique element in $\Psi(\mc V')$ attached to $\lambda\in P$ by our choices. For $\lambda\in P$ and 
$x\in\dom\lambda$ we clearly have
\begin{equation}
\label{welldefinedgerms} \varphi_1(\germ_x \lambda) = \germ_{\varphi_0(x)}\nu(\lambda). 
\end{equation}
Properties~(\apref{R}{nu}{}) are easily checked using the compatibility of $\varphi$ with the structure maps.

It remains to show that the image  of $\varphi_0\vert_{V_i}$ is
contained in  $V'_j$ for some orbifold chart $(V'_j, G'_j, \pi'_j)\in\mathcal V'$. Since $V_i$ is connected, the image $\varphi_0(V_i)$ is connected as well. The connected components of $\Gamma(\mc V')_0$ are exactly the sets $W'$ with $(W',G',\varphi')\in \mc V$. From this the claim follows.
\end{proof}

Proposition~\ref{inducedorbgr} guarantees that each homomorphism
\[ \varphi = (\varphi_0,\varphi_1) \co \Gamma(\mathcal V) \to \Gamma(\mathcal V') \]
induces a representative of an orbifold map $(f, \{\tilde f_i\}_{i\in I},
P,\nu)$ with domain atlas $\mathcal V$, range family contained in $\mathcal
V'$, $\tilde f_i = \varphi_0\vert_{V_i}$, and $f=\alpha_{\mc V'}\circ|\varphi|\circ\alpha_{\mc V}^{-1}$. 
For the pair $(P,\nu)$, Proposition~\ref{inducedorbgr} allows (in general) a whole bunch of choices. On the
other hand, different representatives of an orbifold map may induce the same
groupoid homomorphism. In view of Proposition~\ref{orb_gr} and the proof of Proposition~\ref{inducedorbgr}, the relevant information stored by the pair $(P,\nu)$ are the germs of the elements in $P$ and the via $\nu$ associated germs of elements 
in $\Psi(\mathcal V')$. This observation is the motivation for the equivalence
relation in the following definition.

\begin{defi}\label{Pnu_equiv} Let 
\[
\hat f \sceq (f, \{\tilde f_i\}_{i\in I}, P_1, \nu_1)\quad\text{and}\quad \hat g\sceq (g,\{\tilde g_i\}_{i\in I}, P_2,\nu_2)
\]
be two representatives of orbifold maps with the same domain atlas $\mathcal V$ representing the orbifold structure $\mc U$ on $Q$ and both range families being contained in the orbifold atlas $\mathcal V'$ of $(Q',\mc U')$. Set $\psi\sceq\coprod_{i\in I}\tilde f_i$. 
We say that $\hat f$ is  \textit{equivalent} to $\hat g$ if $f=g$, $\tilde f_i = \tilde g_i$ for all $i\in I$, and 
\[ \germ_{\psi(x)} \nu_1(\lambda_1) = \germ_{\psi(x)}\nu_2(\lambda_2)\]
for all $\lambda_1\in P_1$, $\lambda_2\in P_2$, $x\in\dom \lambda_1 \cap \dom\lambda_2$ with  $\germ_x\lambda_1 = \germ_x\lambda_2$. This defines an equivalence relation. The equivalence class of $\hat f$ will be denoted by $[\hat f]$ or
\[ (f, \{\tilde f_i\}_{i\in I}, [P_1, \nu_1]), \]
or even $\hat f$ if it is clear that we refer to the equivalence class. It is called an \textit{orbifold map with domain atlas $\mathcal V$ and range atlas $\mathcal V'$}, in short \textit{orbifold map with $(\mathcal V,\mathcal V')$} or, if the specific orbifold atlases are not important, a \textit{charted orbifold map}. The set of all orbifold maps with $(\mathcal V,\mathcal V')$ is denoted $\Orbmap(\mathcal V,\mathcal V')$. For convenience we  often denote an element $\hat h\in \Orbmap(\mathcal V,\mathcal V')$ by 
\[
 \mc V \stackrel{\hat h}{\longrightarrow} \mc V'.
\]
\end{defi}

\begin{remark}
Instead of using the pair $[P,\nu]$ in the definition of charted orbifold maps one could also define a variant where $[P,\nu]$ is replaced by a mapping 
\[
 \theta\colon \Germ(\Psi(\mc V)) \to \Germ(\Psi(\mc V')),
\]
which maps germs of transitions to germs such that the equivalence conditions in (\apref{R}{invariant}{}) are satisfied by $\theta$. This would stress the actual ``germ nature'' of $[P,\nu]$ and might simplify the compositions of charted orbifold maps as defined in Construction~\ref{constr_comp} below and support the intuition for its definition. Anyhow, for the reasonings in this manuscript (and other concrete situations) it is more convenient to work with the actual transitions and to have the possibility to choose ``small'' quasi-pseudogroups $P$, for which reason we decided to use the variant with the pair $[P,\nu]$.
\end{remark}

Propositions~\ref{orb_gr} and \ref{inducedorbgr} yield the following statement of which we omit the proof.

\begin{prop}\label{conclusion} Let $\mathcal V$ be a representative of $\mc U$, and $\mathcal V'$ a representative of $\mc U'$. Then the set $\Orbmap(\mathcal V,\mathcal V')$ of all orbifold maps with $(\mathcal V, \mathcal V')$ and the set $\Hom(\Gamma(\mathcal V),\Gamma(\mathcal V'))$ of all homomorphisms from $\Gamma(\mathcal V)$ to $\Gamma(\mathcal V')$  are in bijection. More precisely, the construction in Proposition~\ref{orb_gr} induces a bijection 
\[ F_1\co \Orbmap(\mathcal V,\mathcal V') \to \Hom(\Gamma(\mathcal V),\Gamma(\mathcal V')), \]
and the construction in Proposition~\ref{inducedorbgr} defines a bijection 
\[ F_2\co \Hom(\Gamma(\mathcal V),\Gamma(\mathcal V')) \to \Orbmap(\mathcal V,\mathcal V'),\]
which is inverse to $F_1$.
\end{prop}

\section{The category of reduced orbifolds}\label{sec_redorbcat}

To define an orbifold category where the objects are orbifolds and the morphisms are equivalence classes of charted orbifold maps we have to answer the following questions:
\begin{enumerate}[(i)]
\item \label{q1} When shall two charted orbifold maps be considered as equal? In other words, what shall be the equivalence relation?
\item \label{q2} What shall be the identity morphism of an orbifold?
\item \label{q3} How does one compose $\varphi\in \Orbmap(\mathcal V,\mathcal V')$ and $\psi\in\Orbmap(\mathcal V',\mathcal V'')$?
\item \label{q4} What is the composition in the category? 
\end{enumerate}

The leading idea is that charted orbifold maps are equivalent if and only if they induce the same charted orbifold map on common refinements of the orbifold atlases. Therefore, we will introduce the notion of an induced charted orbifold map.

It turns out that answers to the questions \eqref{q2} and \eqref{q3} naturally extend to answers of \eqref{q1} and \eqref{q4}, and that the arising category has a counterpart in terms of marked atlas groupoids and homomorphisms. We start with the definition of the identity morphism of an orbifold. This definition is based on the idea that the identity morphism of $(Q,\mathcal U)$ shall be represented by a collection of local lifts of $\id_Q$ which locally induce $\id_S$ on some orbifold charts, and that each such collection which satisfies (\apref{R}{liftings}{}) shall be a representative.

\subsection{The identity morphism}\label{liftofid}

\begin{defirem}
Let $(Q,\mathcal U)$ and $(Q',\mathcal U')$ be orbifolds and let $f\co Q\to Q'$ be a continuous map. Suppose that $\tilde f$ is a local lift of $f$ \wrt the orbifold charts $(V,G,\pi)\in \mathcal U$ and $(V', G',\pi') \in \mathcal U'$. Further suppose that
\[
 \lambda \co (W,K,\chi) \to (V,G,\pi)\quad\text{and}\quad  \mu\co (W',K',\chi')\to (V',G',\pi')
\]
are embeddings between orbifold charts in $\mathcal U$ resp.\@  in $\mathcal U'$ such that 
\[
\tilde f(\lambda (W) ) \subseteq \mu(W').
\]
Then the map 
\[ 
\tilde g\sceq \mu^{-1}\circ\tilde f\circ \lambda\co W \to W'
\]
is a local lift of $f$ \wrt $(W,K,\chi)$ and $(W',K',\chi')$. We say that
$\tilde f$ \textit{induces the local lift $\tilde g$ \wrt $\lambda$ and $\mu$},
and we call $\tilde g$ the \textit{induced lift of $f$ \wrt $\tilde f$,
$\lambda$ and $\mu$.}
\[
\xymatrix{
& V \ar[r]^{\tilde f} & V'
\\
& \lambda(W) \ar[r]^{\tilde f\vert_{\lambda(W)}} \ar@{^{(}->}[u] & \mu(W') \ar@{_{(}->}[u]
\\
W \ar[uur]^{\lambda}\ar[rrr]^{\tilde g} \ar@{>->>}[ur]_{\lambda} &&& W' \ar[uul]_{\mu} \ar@{>->>}[ul]^{\mu}
}
\]
\end{defirem}

Suppose that $\tilde f$ is a local lift of the identity $\id_Q$ for some orbifold $(Q,\mathcal U)$. Proposition~\ref{induceslifts} below shows that $\tilde f$ induces the identity on sufficiently small orbifold charts. This means that locally $\tilde f$ is related to the identity itself via embeddings. In particular, $\tilde f$ is a local diffeomorphism. For its proof we need the following lemma, which is easily shown and crucially depends on the finiteness of $G$ and the Hausdorff property of $M$. We refer to \cite[Lemma~B.3]{Schmeding} for the details of the construction of the $G$-stable neighborhood $S$.

\begin{lemma}\label{SIN} 
Let $M$ be a manifold, $G$ a finite subgroup of $\Diff(M)$, and $x\in M$. There
exist arbitrary small open $G$--stable neighborhoods $S$ of $x$. Moreover, one
can choose $S$ so small that $G_S=G_x$, the isotropy group of $x$.
\end{lemma}

\begin{prop}\label{induceslifts}
Let $(Q,\mathcal U)$ be an orbifold and suppose that  $\tilde f$ is a local lift of $\id_Q$ \wrt $(V,G,\pi)$, $(V',G',\pi')\in \mathcal U$. For each $v\in V$ there exist a restriction $(S, G_S, \pi\vert_S)$ of $(V,G,\pi)$ with $v\in S$ and a restriction $(S', (G')_{S'}, \pi'\vert_{S'})$ of $(V', G', \pi')$ such that $\tilde f\vert_S$ is an isomorphism from $(S, G_S, \pi\vert_S)$ to the orbifold chart $(S', (G')_{S'}, \pi'\vert_{S'})$. In particular, $\tilde f\vert_S$ induces the identity $\id_S$ \wrt $\id_S$ and $\left(\tilde f\vert_S\right)^{-1}$.
\end{prop}

\begin{proof}
Let $v\in V$ and set $v'\sceq \tilde f(v)$. Then $\pi(v) = \pi'(v')$. By
compatibility of orbifold charts there exist a restriction $(W,H,\chi)$ of $(V,G,\pi)$ with $v\in W$ and an open
embedding $\lambda \co (W,H,\chi) \to (V',G',\pi')$ such that $\lambda(v) = v'$. Lemma~\ref{SIN} yields an open $H$--stable neighborhood $S$ of $v$ with
$S\subseteq \tilde f^{-1}(\lambda(W))\cap W$.  Let 
\[ \tilde g \sceq \lambda^{-1}\circ \tilde f\vert_S \co S\to W\]
denote the induced lift of $\id_Q$. Since $\chi\circ\tilde g =   \chi$, \cite[Lemma~2.11]{Moerdijk_Mrcun} shows the existence of a unique $h\in H$ such that $\tilde g = h\vert_S$.  Thus $\tilde f\vert_S= \lambda\circ h\vert_S\co S\to \lambda(h(S))$ is a diffeomorphism. In turn 
\[\tilde f\vert_S\co (S, H_S, \chi\vert_S) \to (\tilde f(S),
G'_{\tilde f(S)}, \pi'\vert_{\tilde f(S)})\] 
is an isomorphism of orbifold charts.
\end{proof}

Not each local lift of the identity is a global diffeomorphism, as the following
example shows. 

\begin{example}\label{localliftnotdiffeom} 
Let $Q$ be the open annulus in $\R^2$ with inner radius $1$ and outer radius $2$ centered at the origin, \ie
\[ Q =\{ w\in \C \mid 1 < |w| < 2 \}. \]
The map $\alpha\co Q \to \C \times \R$,
\[ \alpha(w) \sceq \left( \frac{w^2}{|w|^2}, |w| -1 \right) \]
maps $Q$ onto the cylinder $Z \sceq S^1 \times (0,1)$. Note that $\alpha(Q)$ covers $Z$ twice.
Then the map $\beta\co Z\to \C$,
\[ \beta(z, s) \sceq \frac{2}{2-s}z\]
is the linear projection of $Z$ from the point $(0,2)\in \C\times \R$ to the complex plane. The composed map $\tilde f = \beta\circ\alpha\co Q\to \C$, 
\[ \tilde f(w) \sceq \frac{2w^2}{(3-|w|) |w|^2} \]
is smooth (where we consider $\C=\R^2$ as a  $2$-dimensional real manifold) and maps $Q$ onto $Q$. Further it induces a homeomorphism between $Q/\{\pm \id\}$ and $Q$. Hence, if we endow $Q$ with the orbifold atlas 
\[  \big\{ \big(Q, \{\pm \id \}, \tilde f\big), \big(Q,\{\id\}, \id \vphantom{\tilde f}\big) \big\},\]
then $\tilde f$ is a local lift of $\id_Q$ \wrt $(Q,\{\pm\id\},\tilde f)$ and $(Q,\{\id\}, \id)$ but not a global diffeomorphism.
\end{example}

\begin{prop}\label{extending} Let $(Q,\mc U)$ be an orbifold and $\{ \tilde f_i \}_{i\in I}$ a family of local lifts of $\id_Q$  which satisfies (\apref{R}{liftings}{}). Then there exists a pair $(P,\nu)$ such that $(\id_Q, \{\tilde f_i\}_{i\in I}, P, \nu)$ is a representative of an orbifold map on $(Q,\mc U)$. The pair $(P,\nu)$ is unique up to equivalence of representatives of orbifold maps.
\end{prop}

\begin{proof}
This follows immediately from Proposition~\ref{induceslifts} in combination with (\apref{R}{invariant}{}).
\end{proof}

\begin{prop}\label{samestructure}
Let $Q$ be a topological space and suppose that $\mathcal U$ and $\mathcal U'$
are orbifold structures on $Q$. Let 
\[ \hat f = \big( f, \{ \tilde f_i\}_{i\in I}, [P,\nu] \big) \]
be a charted orbifold map for which $f=\id_Q$, the domain atlas $\mathcal V$ is a representative of $\mathcal U$, the range family $\mathcal V'$, which here is an orbifold atlas, is a representative of $\mathcal U'$, and for each $i\in I$, the map $\tilde f_i$ is a local diffeomorphism.
Then $\mathcal U=\mathcal U'$.
\end{prop}

\begin{proof}
Let $(V_i, G_i, \pi_i) \in\mathcal V$, $(V'_j, G'_j, \pi'_j) \in \mathcal V'$ and $x\in V_i$, $y\in V'_j$ such that $\pi_i(x) = \pi'_j(y)$. Since $\tilde f_i \co V_i \to V'_i$ is a local diffeomorphism, there are open neighborhoods $U$ of $x$ in $V_i$ and $U'$ of $\tilde f_i(x)$ in $V'_i$ such that $\tilde f_i\vert_U\co U\to U'$ is a diffeomorphism. We have
\[ \pi'_i\big(\tilde f_i(x)\big) = \pi_i(x) = \pi'_j(y).\]  
Therefore there exist open neighborhoods $W'_1$ of $\tilde f_i(x)$ in $U'$ and $W'_2$ of $y$ in $V'_j$ and a diffeomorphism $h\co W'_1\to W'_2$ satisfying $\pi'_j\circ h= \pi'_i$. Shrinking $U$ shows that $(V_i, G_i, \pi_i)$ and $(V'_j, G'_j, \pi'_j)$ are compatible. Thus $\mc U = \mc U'$.
\end{proof}

The following example shows that the requirement in Proposition~\ref{samestructure} that the local lifts be local diffeomorphisms is essential.

\begin{example}
Recall the orbifolds $(Q,\mc U_i)$, $i=1,2$, from Example~\ref{notcompatible}, the representatives $\mc V_1 \sceq \{V_1\}$ and $\mc V_2 \sceq \{V_2\}$ of $\mc U_1$ resp.\@ $\mc U_2$, and set $g(x) \sceq x^2$ for $x\in (-1,1)$. Then $g$ is a lift of $\id_Q$ \wrt $V_2$ and $V_1$. Further let 
\[
P\sceq \{ \pm \id_{(-1,1)}\}\quad\text{and}\quad \nu( \pm \id_{(-1,1)}) \sceq \id_{(-1,1)}.
\]
Then
$(\id_Q, \{g\}, [P,\nu])$ is an orbifold map with $(\mc V_2, \mc V_1)$ from $(Q,\mc U_2)$ to $(Q,\mc U_1)$, but $\mc U_1 \not= \mc U_2$.
\end{example}

Motivated by Propositions~\ref{extending} and \ref{samestructure} we make the following definition. 

\begin{defi}\label{liftiddef} 
Let $(Q,\mc U)$ be an orbifold and let $\hat f = (f, \{\tilde f_i\}_{i\in I}, [P,\nu])$ be a charted orbifold map whose domain atlas is a representative of $\mc U$. If and only if $f=\id_Q$ and $\tilde f_i$ is a local diffeomorphism for each $i\in I$, we call $\hat f$ a \textit{lift of the identity} $\id_{(Q,\mc U)}$ or a \textit{representative} of $\id_{(Q,\mc U)}$. The set of all lifts of  $\id_{(Q,\mc U)}$ is the \textit{identity morphism} $\id_{(Q,\mc U)}$ of $(Q,\mathcal U)$.
\end{defi}

\subsection{Composition of charted orbifold maps}

\begin{construction}\label{constr_comp}
Let $(Q,\mathcal U)$, $(Q',\mathcal U')$ and $(Q'',\mathcal U'')$ be orbifolds, and 
\begin{align*}
 \mc V \sceq \{ (V_i,G_i,\pi_i) \mid i\in I\},\qquad \mc V' \sceq \{ (V'_j,G'_j,\pi'_j) \mid j\in J \}
\end{align*}
resp.\@ $\mathcal V''$ be representatives for $\mathcal U$, $\mathcal U'$ resp.\@ $\mathcal U''$, where $\mc V$ resp.\@ $\mc V'$ are indexed by $I$ resp.\@ $J$. Suppose  that
\[ \hat f = (f, \{\tilde f_i\}_{i\in I}, [P_f,\nu_f]) \in \Orbmap(\mathcal V, \mathcal V')\]
and
\[ \hat g = (g, \{\tilde g_j\}_{j\in J}, [P_g,\nu_g]) \in \Orbmap(\mathcal V',\mathcal V'')\]
are charted orbifold maps and that $\alpha\co I \to J$ is the unique map such that for each $i\in I$, $\tilde f_i$ is a local lift of $f$ w.r.t.\@ $(V_i,G_i,\pi_i)$ and $(V'_{\alpha(i)}, G'_{\alpha(i)}, \pi'_{\alpha(i)})$.  The composition
\[ \hat g \circ \hat f \sceq \hat h = (h, \{\tilde h_i\}_{i\in I}, [P_h,\nu_h]) \in \Orbmap(\mathcal V,\mathcal V'')\]
is given by $h\sceq g\circ f$ and $\tilde h_i \sceq \tilde g_{\alpha(i)} \circ \tilde f_i$ for all $i\in I$. To construct a representative $(P_h,\nu_h)$ of $[P_h,\nu_h]$ we fix representatives $(P_f,\nu_f)$ and $(P_g,\nu_g)$ of $[P_f,\nu_f]$ and $[P_g,\nu_g]$, respectively. 
The leading idea to define $(P_h,\nu_h)$ is to take $P_h = P_f$ and $\nu_h = \nu_g \circ \nu_f$. But since $\nu_f(\lambda)$ is not necessarily in $P_g$ for $\lambda\in P_f$, the composition $\nu_g\circ\nu_f$ might be ill-defined. In the following we refine this idea.

Let $\mu \in P_f$ and suppose that $\dom\mu \subseteq V_i$ and $\cod\mu\subseteq V_j$ for the orbifold charts $(V_i,G_i,\pi_i)$ and $(V_j,G_j,\pi_j)$ in $\mathcal V$. By (\apref{R}{invariant}{})
\[ \tilde f_j \circ \mu = \nu_f(\mu) \circ \tilde f_i \vert_{\dom \mu},\]
where $\nu_f(\mu) \in \Psi(\mc U')$. By possibly shrinking domains, we may assume that $\nu_f(\mu) \in \Psi(\mc V')$. For $x\in\dom\mu$ we set $y_x \sceq \tilde f_i(x)$, which is an element of $\dom\nu_f(\mu)$. Hence we find (and fix a choice) $\xi_{\mu,x} \in P_g$ with $y_x\in \dom\xi_{\mu,x}$ and an open set $U'_{\mu,x} \subseteq \dom \xi_{\mu,x} \cap \dom \nu_f(\mu)$ such that $y_x \in U'_{\mu,x}$ and
\[
 \xi_{\mu,x}\vert_{U'_{\mu,x}} = \nu_f(\mu)\vert_{U'_{\mu,x}}.
\]
Then we find (and fix) an open set $U_{\mu,x} \subseteq \dom\mu$ with $x\in U_{\mu,x}$ such that $\tilde f_i(U_{\mu,x}) \subseteq U'_{\mu,x}$. By adjusting choices we achieve that for $\mu_1,\mu_2\in P_f$ and $x_1\in\dom\mu_1$, $x_2\in\dom\mu_2$ we either have
\begin{equation}\label{wll1}
\mu_1\vert_{U_{\mu_1,x_1}} \not= \mu_2\vert_{U_{\mu_2,x_2}}\quad\text{or}\quad \xi_{\mu_1,x_1} = \xi_{\mu_2,x_2}.
\end{equation}
Define
\[
 P_h\sceq \big\{ \mu\vert_{U_{\mu,x}} \big\vert\ \mu\in P_f,\ x\in\dom\mu \big\},
\]
which obviously is a quasi-pseudogroup generating $\Psi(\mc V)$, and set
\[
 \nu_h\big( \mu\vert_{U_{\mu,x}} \big) \sceq \nu_g(\xi_{\mu,x})
\]
for $\mu\vert_{U_{\mu,x}}\in P_h$. Property~\eqref{wll1} yields that $\nu_h$ is a well-defined map from $P_h$ to $\Psi(\mc U'')$. One 
easily sees that $\nu_h$ satisfies (\apref{R}{invariant}{}) - (\apref{R}{unit_pg}{}),
and that the equivalence class of $(P_h,\nu_h)$ does not depend on the choices we made for the construction of $P_h$ and $\nu_h$.
\end{construction}

\begin{remark}\label{F1_equivariant}
The construction of the composition of two charted orbifold maps  immediately implies that  the maps $F_1$ and $F_2$ (see Proposition~\ref{conclusion}) are both functorial.
\end{remark}

The following lemma provides the definition of induced charted orbifold map and shows its relation to lifts of the identity.

\begin{lemmadefi}\label{onlyinduced}
Let  $(Q,\mathcal U)$ and $(Q',\mathcal U')$ be orbifolds. Further let 
\begin{align*}
\mathcal V & = \{ (V_i, G_i, \pi_i) \mid i\in I \} \ \text{be a representative of $\mathcal U$, indexed by $I$,}
\\
\mathcal V' & = \{ (V'_l, G'_l, \pi'_l) \mid l\in L\} \ \text{be a representative of $\mathcal U'$, indexed by $L$,}
\\ 
\hat f & = \big( f, \{ \tilde f_i\}_{i\in I}, [P_f,\nu_f]  \big) \in \Orbmap(\mathcal V, \mathcal V'),
\end{align*}
and let $\beta\co I\to L$ be the unique map such that for each $i\in I$, $\tilde f_i$ is a local lift of $f$ w.r.t.\@ $(V_i,G_i,\pi_i)$ and $(V'_{\beta(i)}, G'_{\beta(i)}, \pi'_{\beta(i)})$.
Suppose that we have
\begin{itemize}
\item a representative $\mathcal W = \{ (W_j, H_j, \psi_j) \mid j\in J \}$ of $\mathcal U$, indexed by $J$,
\item a subset $\{ (W'_j, H'_j, \psi'_j) \mid j\in J \}$ of $\mc U'$, indexed by $J$ (not necessarily an orbifold atlas),
\item a map $\alpha\co J\to I$,
\item for each $j\in J$, an embedding
\[ \lambda_j \co \big(W_j, H_j, \psi_j\big) \to \big(V_{\alpha(j)}, G_{\alpha(j)}, \pi_{\alpha(j)}\big),\]
and an embedding
\[ \mu_j \co \big(W'_j, H'_j, \psi'_j\big) \to \big(V'_{\beta(\alpha(j))}, G'_{\beta(\alpha(j))}, \pi'_{\beta(\alpha(j))}\big)\]
such that 
\[ \tilde f_{\alpha(j)}\big( \lambda_j(W_j)\big) \subseteq \mu_j(W'_j).\]
\end{itemize}
For each $j\in J$ set
\[ \tilde h_j \sceq \mu_j^{-1} \circ \tilde f_{\alpha(j)}\circ \lambda_j \co W_j \to W'_j.\]
Then 
\begin{enumerate}[{\rm (i)}]
\item \label{firstident} $\eps\sceq \big(\id_Q, \{ \lambda_j\}_{j\in J}, [P_\eps,\nu_\eps]\big) \in \Orbmap(\mathcal W,\mathcal V)$ (with $[P_\eps,\nu_\eps]$ provided by Proposition~\ref{extending}) is a lift of $\id_{(Q,\mc U)}$.
\item \label{secondident} The set $\{ (W'_j, H'_j, \psi'_j)\mid j\in J\}$ and the family $\{\mu_j\}_{j\in J}$ can be extended to a representative
\[ \mathcal W' = \big\{ (W'_k, H'_k,\psi'_k) \ \big\vert\ k\in K \big\} \]
of $\mathcal U'$ and a family of embeddings $\{\mu_k\}_{k\in K}$ such that
\[ \eps'\sceq \big(\id_{Q'}, \{ \mu_k\}_{k\in K}, [P_{\eps'},\nu_{\eps'}]\big) \in \Orbmap(\mathcal W',\mathcal V') \]
(with $[P_{\eps'},\nu_{\eps'}]$ provided by Proposition~\ref{extending}) is a lift of the identity $\id_{(Q',\mc U')}$.
\item \label{mapind} There is a uniquely determined equivalence class $[P_h,\nu_h]$ such that 
\[ \hat h \sceq (f, \{\tilde h_j\}_{j\in J}, [P_h,\nu_h]) \in \Orbmap(\mathcal W,\mathcal W') \]
and such that the diagram
\[
\xymatrix{ 
& \mathcal V \ar[r]^{\hat f} & \mathcal V'
\\
\mathcal W \ar[ur]^{\eps} \ar[rrr]^{\hat h} &&& \mathcal W' \ar[ul]_{\eps'}
}
\]
commutes.
\end{enumerate}
We say that $\hat h$ is \emph{induced} by $\hat f$.
\end{lemmadefi}

\begin{proof}
\eqref{firstident} is clear by Proposition~\ref{induceslifts} and \ref{extending}. To show that \eqref{secondident} holds we construct one possible extension: Let 
\[ y\in Q' \setminus \bigcup_{j\in J} \psi'_j(W'_j).\]
Then there is a chart $(V', G',\pi') \in \mathcal V'$ such that $y\in\pi'(V')$. Extend the set 
\[
\{ (W'_j, H'_j,\psi'_j) \mid j\in J\}
\]
with $(V',G',\pi')$ and the family $\{ \mu_j\}_{j\in J}$ with $\id_{V'}$. If this is done iteratively, one finally gets an orbifold atlas of $Q'$ as wanted. Then Proposition~\ref{induceslifts} and \ref{extending} yield the remaining claim of \eqref{secondident}.  The following considerations are independent of the specific choices of extensions.
Concerning \eqref{mapind} we remark that each $\tilde h_j$ is obviously a local lift of $f$. Fix a representative $(P_f,\nu_f)$ of $[P_f,\nu_f]$. In the following we construct a pair $(P_h,\nu_h)$ for which $\hat h$ is an orbifold map and the diagram in \eqref{mapind} commutes. It will be clear from the construction that the equivalence class $[P_h,\nu_h]$ is independent of the choice of $(P_f,\nu_f)$ and uniquely determined by the requirement of the commutativity of the diagram. Let $\gamma\in \Psi(\mathcal W)$ and $x\in\dom\gamma$. Possibly shrinking the domain of $\gamma$, we may assume that $\dom \gamma \subseteq W_j$ and $\cod \gamma \subseteq W_k$ for some $j,k\in J$. 
In the following we further shrink the domain of $\gamma$ to be able to define $\nu_h$ as a composition of $\nu_f$ with elements of $\{\mu_j\}_{j\in J}$. Let $y \sceq \lambda_j(x)$. Since
\[ \tilde\gamma\sceq \lambda_k \circ \gamma\circ \left(\lambda_j\vert_{\dom\gamma} \right)^{-1} \co \lambda_j(\dom\gamma) \to \lambda_k(\cod\gamma) \]
is an element of $\Psi(\mathcal V)$, we find $\beta_\gamma \in P_f$ such that $y\in\dom\beta_\gamma$ and $\germ_y\beta_\gamma = \germ_y\tilde \gamma$. Then 
\[ z\sceq \tilde f_{\alpha(j)}(y) \in \dom\nu_f(\beta_\gamma) \cap \mu_j(W'_j).\]
Since
\[
 \nu_f(\beta_\gamma)(z) = \tilde f_{\alpha(k)}(\beta_\gamma(y)) \in \mu_k(W'_k),
\]
the set
\[ U' \sceq \dom\nu_f(\beta_\gamma) \cap \mu_j(W'_j) \cap \nu_f(\beta_\gamma)^{-1}(\mu_k(W'_k)) \]
is an open neighborhood of $z$. Define
\[ U_1 \sceq \{ w\in\dom\beta_\gamma \cap \lambda_j(\dom\gamma) \mid \germ_w\beta_\gamma = \germ_w\tilde\gamma \},\]
which is an open neighborhood of $y$. Then also 
\[ U \sceq U_1 \cap \tilde f_{\alpha(j)}^{-1}(U') \]
is an open neighborhood of $y$. We fix an open neighborhood $U_{\gamma,x}$ of $x$ in $\lambda_j^{-1}(U)$. Further we suppose that for $\gamma_1,\gamma_2\in \Psi(\mc W)$, $x_1\in\dom\gamma_1$, $x_2\in\dom\gamma_2$,  we either have
\begin{equation}\label{wll2}
\gamma_1\vert_{U_{\gamma_1,x_1}} \not= \gamma_2\vert_{U_{\gamma_2,x_2}} \quad\text{or}\quad \nu_f(\beta_{\gamma_1}) = \nu_f(\beta_{\gamma_2}).
\end{equation}
Then we define
\[
 P_h \sceq \big\{ \gamma\vert_{U_{\gamma,x}} \big\vert\ \gamma\in\Psi(\mc W),\ x\in\dom\gamma \big\}
\]
and set
\[
 \nu_h\big(\gamma\vert_{U_{\gamma,x}}\big) \sceq \mu_k^{-1}\circ \nu_f(\beta_\gamma) \circ \mu_j
\]
for $\gamma\vert_{U_{\gamma,x}} \in P_h$ with $x\in W_j$ and $\gamma(x) \in W_k$ ($j,k\in J$). The map $\nu_h\co P_h\to \Psi(\mc W')$ is well-defined by \eqref{wll2}. One easily checks that $(P_h,\nu_h)$ satisfies all requirements of \eqref{mapind}.
\end{proof}

We consider two charted orbifold maps as equivalent if they induce the same
charted orbifold map on common refinements of the orbifold atlases. The following
definition provides a precise specification of this idea.

\begin{defi}\label{mapequiv}
Let $(Q,\mathcal U)$ and $(Q',\mathcal U')$ be orbifolds. Further let $\mathcal V_1, \mathcal V_2$ be representatives of $\mathcal U$, and $\mathcal V'_1, \mathcal V'_2$ be representatives of $\mathcal U'$. Suppose that $\hat f_1\in \Orbmap(\mathcal V_1,\mathcal V'_1)$ and $\hat f_2\in \Orbmap(\mathcal V_2,\mathcal V'_2)$. We call $\hat f_1$ and $\hat f_2$ \textit{equivalent} ($\hat f_1 \sim \hat f_2$) if there are
a representative $\mathcal W$ of $\mathcal U$, a representative $\mathcal W'$ of $\mathcal U'$, $\eps_1\in \Orbmap(\mathcal W,\mathcal V_1)$, $\eps_2\in \Orbmap(\mathcal W,\mathcal V_2)$ lifts of $\id_{(Q,\mc U)}$, $\eps'_1\in \Orbmap(\mathcal W',\mathcal V'_1)$, $\eps'_2\in \Orbmap(\mathcal W',\mathcal V'_2)$ lifts of $\id_{(Q',\mc U')}$, and a map $\hat h\in \Orbmap(\mathcal W,\mathcal W')$ 
such that the diagram
\[
\xymatrix{
& \mathcal V_1 \ar[r]^{\hat f_1} & \mathcal V'_1
\\
\mathcal W \ar[ur]^{\eps_1} \ar[rrr]^{\hat h} \ar[dr]_{\eps_2} &&& \mathcal W' \ar[ul]_{\eps'_1} \ar[dl]^{\eps'_2}
\\
& \mathcal V_2 \ar[r]_{\hat f_2} & \mathcal V'_2
}
\]
commutes.
\end{defi}

Proposition~\ref{mapwell} below shows that $\sim$ is indeed an equivalence
relation. For its proof we need the following two lemmas. The first lemma discusses
how local lifts  which belong to the same charted orbifold map are related to
each other. The second lemma shows that two charted orbifold maps which are induced
from the same charted orbifold map induce the same charted orbifold map on
common refinements of orbifold atlases. This means that $\sim$ satisfies the so-called
diamond property.

\begin{lemma}\label{welldefinedll}
Let $(Q,\mathcal U)$ and $(Q',\mathcal U')$ be orbifolds and let 
\[ \hat f \sceq (f, \{ \tilde f_i \}_{i\in I}, [P,\nu]) \in \Orbmap(\mathcal V, \mathcal V')\]
be a charted orbifold map where $\mathcal V$ is a representative of $\mathcal U$ and $\mathcal V'$ one of $\mathcal U'$. Suppose that we have orbifold charts $(V_a, G_a, \pi_a), (V_b, G_b, \pi_b) \in \mathcal V$ and points $x_a\in V_a$, $x_b\in V_b$ such that $\pi_a(x_a) = \pi_b(x_b)$. Then there are arbitrarily small orbifold charts $(W,K,\chi)\in \mathcal U$, $(W',K',\chi')\in\mathcal U'$ and embeddings
\begin{align*}
\lambda & \co (W,K,\chi) \to (V_a, G_a, \pi_a)
\\
\lambda' & \co (W',K',\chi') \to (V'_a, G'_a, \pi'_a)
\\
\mu & \co (W,K,\chi) \to (V_b, G_b,\pi_b)
\\
\mu' & \co (W', K',\chi') \to (V'_b, G'_b, \pi'_b)
\end{align*}
with $x_a\in \lambda(W)$ and $x_b\in \mu(W)$ such that the induced lift $\tilde g$ of $f$ \wrt $\tilde f_a, \lambda, \lambda'$ coincides with the one induced by $\tilde f_b$, $\mu$, $\mu'$. In other words, the diagram
\[
\xymatrix{
& V_a \ar[r]^{\tilde f_a} & V'_a 
\\
W \ar[ur]^{\lambda} \ar[rrr]^{\tilde g} \ar[dr]_{\mu} &&& W' \ar[ul]_{\lambda'} \ar[dl]^{\mu'}
\\
& V_b \ar[r]^{\tilde f_b} & V'_b
}
\]
commutes.
\end{lemma}

\begin{proof}
By compatibility of orbifold charts we find an arbitrarily small restriction $(W,K,\chi)$ of $(V_a, G_a,\pi_a)$ with $x_a\in W$ and an embedding 
\[
\mu\co (W,K,\chi) \to (V_b, G_b, \pi_b)
\]
such that $\mu(x_a) = x_b$. Then $\mu\co W\to \mu(W)$ is an element of $\Psi(\mathcal V)$. Fix a representative $(P,\nu)$ of $[P,\nu]$. Hence there is $\gamma\in P$ with $x_a\in \dom\gamma$ and an open neighborhood $U$ of $x_a$ such that $U\subseteq \dom \gamma \cap W$ and
\[ \mu\vert_U = \gamma\vert_U.\]
W.l.o.g.\@, $\gamma = \mu$. Property~(\apref{R}{invariant}{}) yields that
\[ \nu(\mu)\circ \tilde f_a\vert_W = \tilde f_b\circ\mu.\]
By shrinking the domain of $\nu(\mu)$, we can achieve that $\cod \nu(\mu)
\subseteq V'_b$ and still $\tilde f_a(W)\subseteq \dom \nu(\mu) \seqc W'$. With
$\mu'\sceq \nu(\mu)$ it follows 
\[ \tilde f_b (\mu(W)) = \mu'(\tilde f_a(W)) \subseteq \mu'(W') \] 
and further
\[ \tilde f_a\vert_W = (\mu')^{-1}\circ \tilde f_b \circ \mu.\]
This proves the claim.
\end{proof}

\begin{lemma}\label{fortrans}
Let $(Q,\mathcal U)$ and $(Q',\mathcal U')$ be orbifolds, $\mathcal V$ a
representative of $\mathcal U$, and $\mathcal V'$ one of $\mathcal U'$. Further
let $\hat f \in \Orbmap(\mathcal V,\mathcal V')$. Suppose that 
 $\hat h \in \Orbmap(\mathcal W_1, \mathcal W'_1)$ and
$\hat g \in \Orbmap(\mathcal W_2, \mathcal W'_2)$ are both induced by $\hat f$. Then
we find a representative $\mathcal W$ of $\mathcal U$ and charted orbifold
maps $\eps_1 \in \Orbmap(\mathcal W, \mathcal W_1)$, $\eps_2\in \Orbmap(\mathcal
W,\mathcal W_2)$ which are lifts of $\id_{(Q,\mc U)}$, and a representative $\mathcal W'$
of $\mathcal U'$ and charted orbifold maps $\eps'_1 \in \Orbmap(\mathcal
W',\mathcal W'_1)$, $\eps'_2\in \Orbmap(\mathcal W',\mathcal W'_2)$ which are
lifts of $\id_{(Q',\mc U')}$, and a charted orbifold map $\hat k \in \Orbmap(\mathcal
W,\mathcal W')$ such that the diagram
\[
\xymatrix{
 & \mathcal W_1 \ar[r]^{\hat h} & \mathcal W'_1
\\
\mathcal W \ar[ur]^{\eps_1} \ar[rrr]^{\hat k}  \ar[dr]_{\eps_2} &&& \mathcal W' \ar[ul]_{\eps'_1} \ar[dl]^{\eps'_2}
\\
& \mathcal W_2 \ar[r]^{\hat g} & \mathcal W'_2
}
\]
commutes. If the orbifolds are second-countable, we can choose $\mc W, \mc W'$ to be countable.
\end{lemma}

\begin{proof}
Suppose that $\hat f  = (f, \{ \tilde f_a\}_{a\in A}, [P_f,\nu_f])$, $\hat h = (f, \{\tilde h_i\}_{i\in I}, [P_h,\nu_h])$ and $\hat g = (f, \{\tilde g_j\}_{j\in J}, [P_g,\nu_g])$. Let
\begin{align*}
\mathcal W_1 & \sceq \{ (W_{1,i}, H_{1,i}, \psi_{1,i} ) \mid i \in I \}, \text{ indexed by $I$,}
\\
\mathcal W'_1 & \sceq \{ (W'_{1,k}, H'_{1,k}, \psi'_{1,k}) \mid k\in K \}, \text{ indexed by $K$,}
\\
\mathcal W_2 & \sceq \{ (W_{2,j}, H_{2,j}, \psi_{2,j}) \mid j\in J \}, \text{ indexed by $J$,}
\\
\mathcal W'_2 & \sceq \{ (W'_{2,l}, H'_{2,l}, \psi'_{2,l}) \mid l\in L \}, \text{ indexed by $L$,}
\end{align*}
and let $\alpha_1\co I\to K$ resp.\@ $\alpha_2\co J\to L$ be the map such that for each $i\in I$, $\tilde h_i$ is a local lift of $f$ w.r.t.\@ $(W_{1,i}, H_{1,i}, \psi_{1,i})$ and $(W'_{1,\alpha_1(i)}, H'_{1,\alpha_1(i)}, \psi'_{1,\alpha_1(i)})$ resp.\@ for each $j\in J$, $\tilde g_j$ is a local lift of $f$ w.r.t.\@ $(W_{2,j}, H_{2,j}, \psi_{2,j})$ and $(W'_{2,\alpha_2(j)}, H'_{2,\alpha_2(j)}, \psi'_{2,\alpha_2(j)})$.
Further let
\begin{align*}
\delta_1 & = (\id_Q, \{ \lambda_{1,i} \}_{i\in I}, [R_1,\sigma_1]) \in \Orbmap(\mathcal W_1,\mathcal V),
\\
\delta_2 & = (\id_Q, \{ \lambda_{2,j} \}_{j\in J}, [R_2,\sigma_2]) \in \Orbmap(\mathcal W_2,\mathcal V)
\intertext{be lifts of $\id_{(Q,\mc U)}$ and}
\delta'_1 & = (\id_{Q'}, \{ \mu_{1,k} \}_{k\in K}, [R'_1,\sigma'_1]) \in \Orbmap(\mathcal W'_1, \mathcal V'),
\\
\delta'_2 & = (\id_{Q'}, \{ \mu_{2,l} \}_{l\in L}, [R'_2, \sigma'_2]) \in \Orbmap(\mathcal W'_2, \mathcal V')
\end{align*}
be lifts of  $\id_{(Q',\mc U')}$ such that $\hat f \circ\delta_1 = \delta'_1\circ \hat h$ and $\hat f\circ\delta_2 = \hat g\circ\delta'_2$.
W.l.o.g.\@ we assume that all $\lambda_{1,i}$, $\mu_{1,k}$, $\lambda_{2,j}$ and $\mu_{2,l}$ are embeddings.
We will use Lemma~\ref{onlyinduced} to show the existence of $\hat k$. More precisely, we attach to each $q\in Q$ an orbifold chart $(W_q,H_q,\psi_q) \in \mc U$ with $q\in \psi_q(W_q)$ and an orbifold chart $(W'_q, H'_q, \psi'_q) \in \mc U'$ with $f(q) \in \psi'_q(W'_q)$. We consider orbifold charts defined for distinct $q$ to be distinct. In this way, we get a representative
\begin{equation}\label{defW}
\mc W \sceq \{ (W_q,H_q,\psi_q) \mid q\in Q\}                                                                                                                                                                                                                                                                                                                                                                                             
\end{equation}                                                                                                                                                                                                                                                                                                                                                                                                
of $\mc U$ which is indexed by $Q$, and a subset $\{ (W'_q,H'_q,\psi'_q) \mid q\in Q\}$ of $\mc U'$,  indexed by $Q$ as well. Moreover, we will find maps $\beta_1 \co Q \to I$ and $\beta_2\co Q\to J$ and embeddings
\begin{align*}
\xi_{1,q} & \co \big(W_q,H_q,\psi_q\big) \to \big(W_{1,\beta_1(q)}, H_{1,\beta_1(q)}, \psi_{1,\beta_1(q)}\big)
\\
\xi_{2,q} & \co \big(W_q,H_q,\psi_q\big) \to \big(W_{2,\beta_2(q)}, H_{2,\beta_2(q)}, \psi_{2,\beta_2(q)}\big)
\\
\chi_{1,q} & \co \big(W'_q,H'_q,\psi'_q\big) \to \big(W'_{1,\alpha_1(\beta_1(q))}, H'_{1,\alpha_1(\beta_1(q))}, \psi'_{1,\alpha_1(\beta_1(q))}\big)
\\
\chi_{2,q} & \co \big(W'_q,H'_q,\psi'_q\big) \to \big(W'_{2,\alpha_2(\beta_2(q))}, H'_{2,\alpha_2(\beta_2(q))}, \psi'_{2,\alpha_2(\beta_2(q))}\big)
\end{align*}
such that for each $q\in Q$ the local lift $\tilde k_q$ of $f$ induced by $\tilde h_{\beta_1(q)}$, $\xi_{1,q}$ and $\chi_{1,q}$ coincides with the one induced by $\tilde g_{\beta_2(q)}$, $\xi_{2,q}$ and $\chi_{2,q}$. Then Lemma~\ref{onlyinduced} shows that $\hat h$ resp.\@ $\hat g$ induces a charted orbifold map $(f, \{\tilde k_q\}_{q\in Q}, [P_1,\nu_1])$ resp.\@ $(f,\{\tilde k_q\}_{q\in Q}, [P_2,\nu_2])$. It then remains to show that we can choose all the embeddings $\xi_{1,q}, \xi_{2,q}, \chi_{1,q}, \chi_{2,q}$ in a way that $[P_1,\nu_1]$ equals $[P_2,\nu_2]$. 

Let $q\in Q$. We fix $i\in I$ such that $q\in \psi_{1,i}(W_{1,i})$ and we pick $w_1\in W_{1,i}$ with $q=\psi_{1,i}(w_1)$.
We set $\beta_1(q) \sceq i$. Further we fix $j\in J$ such that $q\in \psi_{2,j}(W_{2,j})$, and pick an element $w_2\in W_{2,j}$ with
$q=\psi_{2,j}(w_2)$. We set $\beta_2(q) \sceq j$. By Lemma~\ref{welldefinedll} we find orbifold charts $(W_q, H_q, \psi_q) \in \mathcal U$ with $q\in
\psi_q(W_q)$, say $q=\psi_q(w_q)$, and $(W'_q,H'_q,\psi'_q)\in\mathcal U'$ with
$f(q) \in \psi'_q(W'_q)$ and embeddings $\xi_{1,q}$, $\xi_{2,q}$, $\chi_{1,q}$,
$\chi_{2,q}$ with $w_1=\xi_{1,q}(w_q)$, $w_2=\xi_{2,q}(w_q)$, and a local lift $\tilde k_q$ of $f$ such that the diagram
\[
\xymatrix{ 
& \lambda_{1,\beta_1(q)}\big(W_{1,\beta_1(q)}\big) \ar[r]^{\tilde f_{\beta_1(q)}} & \mu_{1,\alpha_1(\beta_1(q))}\big(W'_{1,\alpha_1(\beta_1(q))}\big)
\\
& W_{1,\beta_1(q)} \ar[u]^{\lambda_{1,\beta_1(q)}} \ar[r]^{\tilde h_{\beta_1(q)}} & W'_{1,\alpha_1(\beta_1(q))} \ar[u]_{\mu_{1,\alpha_1(\beta_1(q))}}
\\
W_q \ar[ur]^{\xi_{1,q}} \ar[dr]_{\xi_{2,q}} \ar[rrr]^{\tilde k_q} &&& W'_q \ar[ul]_{\chi_{1,q}} \ar[dl]^{\chi_{2,q}}
\\
& W_{2,\beta_2(q)} \ar[r]^{\tilde g_{\beta_2(q)}} \ar[d]_{\lambda_{2,\beta_2(q)}} & W'_{2,\alpha_2(\beta_2(q))} \ar[d]^{\mu_{2,\alpha_2(\beta_2(q))}}
\\
& \lambda_{2,\beta_2(q)}\big(W_{2,\beta_2(q)}\big) \ar[r]^{\tilde f_{\beta_2(q)}} & \mu_{2,\alpha_2(\beta_2(q))}\big(W'_{2,\alpha_2(\beta_2(q))}\big)
}
\]
commutes. We may assume that $\xi_{1,q}=\id$ and $\chi_{1,q}=\id$. Now
\[ \eta\sceq\lambda_{2,\beta_2(q)}\circ\xi_{2,q}\circ\lambda_{1,\beta_1(q)}^{-1} \co \lambda_{1,\beta_1(q)}(W_q) \to \lambda_{2,\beta_2(q)}\big(\xi_{2,q}(W_q)\big) \]
is an element of $\Psi(\mathcal V)$ with $y\sceq \lambda_{1,\beta_1(q)}(\xi_{1,q}(w_q))$ in its domain. 
We pick a representative $(P_f,\nu_f)$ of $[P_f,\nu_f]$. Then there is an element $\gamma\in P_f$ with $y\in\dom\gamma$ and an open neighborhood $U$ of $y$ such that $U\subseteq \dom\gamma\cap\dom\eta$ and $\eta\vert_U = \gamma\vert_U$.
By (\apref{R}{invariant}{}),
\[
 \nu_f(\gamma)\circ \tilde f_{\beta_1(q)}\vert_U = \tilde f_{\beta_2(q)}\circ\gamma\vert_U = \tilde f_{\beta_2(q)}\circ \eta\vert_U.
\]
The map
\[
 \mu \sceq \mu_{2,\alpha_2(\beta_2(q))} \circ \chi_{2,q} \circ \mu^{-1}_{1,\alpha_1(\beta_1(q))} \co \mu_{1,\alpha_1(\beta_1(q))}(W'_q)\to \mu_{2,\alpha_2(\beta_2(q))}\big(\chi_{2,q}(W'_q)\big)
\]
is a diffeomorphism as well. Further there exists an open neighborhood $V$ of $y$ such that 
\[
 \tilde f_{\beta_2(q)}\circ \eta\vert_V = \mu\circ \tilde f_{\beta_1(q)}\vert_V.
\]
Hence 
\[
 \nu_f(\gamma)\circ\tilde f_{\beta_1(q)} = \mu \circ \tilde f_{\beta_1(q)}
\]
on some neighborhood of $y$. Therefore, after possibly shrinking $W_q$, we can redefine $W'_q$, $\chi_{2,q}$ and $\tilde k_q$  such that 
\begin{equation}\label{nubeta}
\chi_{2,q} = \mu_{2,\alpha_2(\beta_2(q))}^{-1} \circ \nu_f(\gamma) \circ \mu_{1,\alpha_1(\beta_1(q))}\vert_{W'_q}.
\end{equation}
We remark that this redefinition might be quite serious if $\tilde f_{\beta_1(q)}$ and hence $\tilde h_{\beta_1(q)}$, $\tilde g_{\beta_2(q)}$ and $\tilde f_{\beta_2(q)}$ are highly non-injective. But since these maps all behave in the same way, we may perform the changes without running into problems. Let $\mc W$ be defined by \eqref{defW}. Lemma~\ref{onlyinduced}, more precisely its proof, shows that $\hat h$ resp.\@ $\hat g$ induces the orbifold maps 
\[
 \hat k_1 = (f, \{\tilde k_q\}_{q\in Q}, [P_1,\nu_1]) \quad\text{resp.}\quad \hat k_2 = (f, \{\tilde k_q\}_{q\in Q}, [P_2,\nu_2])
\]
with $(\mc W,\mc W')$, where $\mc W'$ is a representative of $\mc U'$ which contains the set 
\[
\{ (W'_q,H'_q,\psi'_q) \mid q\in Q\}
\]
(the proof of Lemma~\ref{onlyinduced} shows that we can indeed have the same $\mc W'$ for $\hat k_1$ and $\hat k_2$).

It remains to show that $[P_1,\nu_1] = [P_2,\nu_2]$. Recall from Lemma~\ref{onlyinduced} that $[P_1,\nu_1]$ is uniquely determined by $\hat h$, $\{\xi_{1,q}\}_{q\in Q}$ and $\{\chi_{1,q}\}_{q\in Q}$, and analogously for $[P_2,\nu_2]$. Alternatively, we may consider $\hat k_1$ and $\hat k_2$ to be induced by $\hat f$. Thus, $[P_1,\nu_1]$ is uniquely determined by $\hat f$, $\{\lambda_{1,\beta_1(q)}\circ\xi_{1,q}\}_{q\in Q}$ and $\{\mu_{1,\alpha_1(\beta_1(q))}\circ \chi_{1,q}\}_{q\in Q}$, and $[P_2,\nu_2]$ is uniquely determined by $\hat f$, $\{\lambda_{2,\beta_2(q)}\circ\xi_{2,q}\}_{q\in Q}$ and $\{\mu_{2,\alpha_2(\beta_2(q))}\circ\chi_{2,q}\}_{q\in Q}$. We fix a representative $(P_f,\nu_f)$ of $[P_f,\nu_f]$. Let $\gamma$ be a change of charts in $\Psi(\mathcal W)$ and $x\in \dom \gamma$. Suppose $\dom\gamma \subseteq W_p$ and $\cod\gamma\subseteq W_q$. Using the same arguments and notation as in the proof of Lemma~\ref{onlyinduced} (without discussing the necessary shrinking of domains, since we are only 
interested in equality in a neighborhood of $x$) we have
\begin{align*}
\beta_h & =\lambda_{1,\beta_1(q)}\circ \gamma\circ\lambda_{1,\beta_1(p)}^{-1},
\\
\beta_g & = \lambda_{2,\beta_2(q)}\circ\xi_{2,q}\circ \gamma\circ\xi_{2,p}^{-1}\circ \lambda_{2,\beta_2(p)}^{-1},
\\
\nu_1(\gamma) &= \mu_{1,\alpha_1(\beta_1(q))}^{-1}\circ\nu_f(\beta_h) \circ \mu_{1,\alpha_1(\beta_1(p))},
\\
\nu_2(\gamma) &= \chi_{2,q}^{-1}\circ\mu_{2,\alpha_2(\beta_2(q))}^{-1}\circ\nu_f(\beta_g) \circ \mu_{2,\alpha_2(\beta_2(p))}\circ \chi_{2,p}.
\end{align*}
Hence
\[ \beta_g= \lambda_{2,\beta_2(q)}\circ\xi_{2,q}\circ\lambda_{1,\beta_1(q)}^{-1}\circ\beta_h\circ\lambda_{1,\beta_1(p)} \circ \xi_{2,p}^{-1}\circ\lambda_{2,\beta_2(p)}^{-1}.\]
Definition~\eqref{nubeta} shows that 
\[
 \nu_f( \lambda_{2,\beta_2(q)} \circ \xi_{2,q} \circ \xi^{-1}_{1,q} \circ \lambda^{-1}_{1,\beta_1(q)}) = \mu_{2,\alpha_2(\beta_2(q))} \circ \chi_{2,q} \circ \mu^{-1}_{1,\alpha_1(\beta_1(q))}.
\]
Then
\begin{align*}
\nu_2(\gamma) & = \mu^{-1}_{1,\alpha_1(\beta_1(q))} \circ \nu_f(\beta_h)\circ \mu_{1,\alpha_1(\beta_1(p))} = \nu_1(\gamma).
\end{align*}

Hence the induced equivalence classes $[P_1,\nu_1]$ and $[P_2,\nu_2]$ indeed
coincide. The lift $\eps_1$ of $\id_{(Q,\mc U)}$ is given by the family $\{ \xi_{1,q}
\}_{q\in Q}$, the lift $\eps_2$ by $\{\xi_{2,q}\}_{q\in Q}$, the lift $\eps'_1$
of $\id_{(Q',\mc U')}$ is any extension of $\{ \chi_{1,q}\}_{q\in Q}$, and the
lift $\eps'_2$ is any extension of $\{ \chi_{2,q}\}_{q\in Q}$.

If the orbifolds are second-countable, we find countable subfamilies of $\mc W$ and $\mc W'$ which are themselves orbifold atlases and restricted to which the orbifold map $\hat k$ still satisfies the statement of the lemma. Alternatively, we could have started the construction with an appropriate countable subset of $Q$.
\end{proof}

The following proposition now follows immediately.

\begin{prop}\label{mapwell}
The relation $\sim$ from Definition~\ref{mapequiv} is an equivalence relation. 
\end{prop}

The equivalence class of a charted orbifold map $\hat f$ with respect to the equivalence from Definition~\ref{mapequiv} is denoted by $[\hat f]$. It will always be clear from context whether $\hat f$ is a charted orbifold map and $[\hat f]$ denotes an equivalence class of charted orbifold maps, or $\hat f$ is a representative of an orbifold map and $[\hat f]$ denotes an equivalence class of representatives, that is, a charted orbifold map (cf.\@ Definition~\ref{Pnu_equiv}).

\subsection{The orbifold category}\label{redorbcat}

Now we can define the category of reduced orbifolds. 

\begin{defi}\label{defredorbcat} The category $\Orbmap$ of reduced (paracompact resp.\@ second-countable) orbifolds is defined as follows: 
In both cases, the class of objects is the class of orbifolds. For two paracompact orbifolds $(Q,\mc U)$
and $(Q',\mc U')$, the morphisms (\textit{orbifold maps}) from $(Q,\mathcal U)$ to $(Q',\mathcal U')$ are
the equivalence classes $[\hat f]$ of all charted orbifold maps $\hat f \in
\Orbmap(\mathcal V, \mathcal V')$ where $\mathcal V$ is any representative of
$\mathcal U$, and $\mathcal V'$ is any representative of $\mathcal U'$, that is
\begin{align*}
 \Morph\big( (Q,\mathcal U), (Q',\mathcal U') \big) \sceq
 \left\{ \big[\hat f\big] \left\vert\   
\begin{minipage}{5cm}
$\hat f\in \Orbmap(\mathcal V, \mathcal V')$,\ $\mathcal V$ repr.\@ of
$\mathcal U$,\\ $\mathcal V'$ representative of $\mathcal U'$ 
\end{minipage} 
\right.\right\}.
\end{align*}
For second-countable orbifolds we restrict to countable representatives of the orbifold structures. We now describe the composition in $\Orbmap$. For this let 
\[
[\hat f] \in \Morph\big( (Q,\mathcal U), (Q',\mathcal U')\big)\quad\text{and}\quad[\hat g] \in \Morph\big(
(Q',\mathcal U'), (Q'',\mathcal U'')\big)
\]
be orbifold maps. Choose representatives $\hat f\in
\Orbmap(\mathcal V, \mathcal V')$ of $[\hat f]$ and $\hat g\in \Orbmap(\mathcal
W',\mathcal W'')$ of $[\hat g]$. Then find (countable, if the orbifolds are second-countable) representatives $\mathcal K$,
$\mathcal K'$, $\mathcal K''$ of $\mathcal U$, $\mathcal U'$, $\mathcal U''$,
\resp, and lifts of identity $\eps\in \Orbmap(\mathcal K, \mathcal V)$,
$\eps'_1\in \Orbmap(\mathcal K',\mathcal V')$, $\eps'_2\in \Orbmap(\mathcal
K',\mathcal W')$, $\eps'' \in \Orbmap(\mathcal K'', \mathcal W'')$ and two charted orbifold maps $\hat h \in \Orbmap(\mathcal K, \mathcal K')$, $\hat k \in
\Orbmap(\mathcal K',\mathcal K'')$ such that the diagram
\[
\xymatrix{ 
& \mathcal V \ar[r]^{\hat f} & \mathcal V' && \mathcal W' \ar[r]^{\hat g} & \mathcal W''
\\
\mathcal K \ar[ur]^{\eps} \ar[rrr]^{\hat h} &&& \mathcal K' \ar[ul]_{\eps'_1} \ar[ur]^{\eps'_2} \ar[rrr]^{\hat k} &&& \mathcal K'' \ar[ul]_{\eps''}
}
\]
commutes. The composition of $[\hat g]$ and $[\hat f]$ is defined to be \[ [\hat g] \circ
[\hat f] \sceq [ \hat k \circ \hat h].\]
\end{defi}

The following lemma shows that this composition is always possible, Proposition~\ref{compositiongood} below that it is well-defined.

\begin{lemma} \label{compwd}
Let $(Q,\mathcal U)$, $(Q',\mathcal U')$ and $(Q'',\mathcal U'')$ be orbifolds.
Further let $\mathcal V$ be a representative of $\mathcal U$, $\mathcal V'$ and
$\mathcal W'$ be representatives of $\mathcal U'$, and $\mathcal W''$ a
representative of $\mathcal U''$. Suppose that $\hat f \in \Orbmap(\mathcal
V,\mathcal V')$ and $\hat g\in\Orbmap(\mathcal W',\mathcal W'')$. Then there exist representatives $\mathcal K$ of $\mathcal U$,
$\mathcal K'$ of $\mathcal U'$, $\mathcal K''$ of $\mathcal U''$, lifts of the
respective identities $\eps \in \Orbmap(\mathcal K,\mathcal V)$, $\eta_1 \in
\Orbmap(\mathcal K',\mathcal V')$, $\eta_2\in \Orbmap(\mathcal K',\mathcal W')$,
$\delta\in \Orbmap(\mathcal K'',\mathcal W'')$, and charted orbifold maps $\hat
h\in \Orbmap(\mathcal K,\mathcal K')$, $\hat k\in \Orbmap(\mathcal K',\mathcal
K'')$ such that the diagram
\[
\xymatrix{ 
& \mathcal V \ar[r]^{\hat f} & \mathcal V' && \mathcal W' \ar[r]^{\hat g} & \mathcal W''
\\
\mathcal K \ar[rrr]^{\hat h} \ar[ur]^{\eps} &&& \mathcal K' \ar[rrr]^{\hat k} \ar[ul]_{\eta_1} \ar[ur]^{\eta_2} &&& \mathcal K'' \ar[ul]_{\delta}
}
\]
commutes. For second-countable orbifolds we can choose $\mc K, \mc K'$ and $\mc K''$ to be countable.
\end{lemma}

\begin{proof} Let $\hat f  = \big(f, \{ \tilde f_i\}_{i\in I}, [P_f,\nu_f] \big)$ and $\hat g = \big(g, \{ \tilde g_j\}_{j\in J}, [P_g,\nu_g] \big)$.
Suppose that 
\begin{align*}
\mathcal V & = \{ (V_i, G_i, \pi_i) \mid i\in I\}, \text{ indexed by $I$,}
\\
\mathcal V' & = \{ (V'_c, G'_c, \pi'_c) \mid c\in C\}, \text{ indexed by $C$,}
\\
\mathcal W' & = \{ (W'_j, H'_j, \psi'_j) \mid j\in J\}, \text{ indexed by $J$,}
\\
\mathcal W'' & = \{ (W''_d, H''_d, \psi''_d) \mid d\in D \}, \text{ indexed by $D$.}
\end{align*}
Let $\tau\co I\to C$ be the map such that for each $i\in I$, $\tilde f_i$ is a local lift of $f$ w.r.t.\@ $(V_i,G_i,\pi_i)$ and $(V'_{\tau(i)}, G'_{\tau(i)}, \pi'_{\tau(i)})$, and $\nu\co J\to D$ the map such that for each $j\in J$, $\tilde g_j$ is a local lift of $g$ w.r.t.\@ $(W'_j, H'_j, \psi'_j)$ and $(W''_{\nu(j)}, H''_{\nu(j)}, \psi''_{\nu(j)})$. By Lemma~\ref{onlyinduced} it suffices to find
\begin{itemize}
\item a representative $\mathcal K = \{ (K_a, L_a, \chi_a) \mid a\in A\}$ of $\mathcal U$, indexed by $A$,
\item a representative $\mathcal K' = \{ (K'_b, L'_b, \chi'_b) \mid b\in B\}$ of $\mathcal U'$, indexed by $B$,
\item a subset $\{ (K''_b, L''_b, \chi''_b) \mid b\in B\}$ of $\mc U''$, indexed by (the same set) $B$,
\item a map $\alpha\co A\to I$,
\item an injective map $\beta\co A\to B$,
\item for each $a\in A$, an embedding
\[ \lambda_a \co (K_a, L_a, \chi_a) \to (V_{\alpha(a)}, G_{\alpha(a)}, \pi_{\alpha(a)}) \]
and an embedding
\[ \mu_a \co (K'_{\beta(a)}, L'_{\beta(a)}, \chi'_{\beta(a)}) \to (V'_{\tau(\alpha(a))}, G'_{\tau(\alpha(a))}, \pi'_{\tau(\alpha(a))})\]
such that 
\[ \tilde f_{\alpha(a)}\big( \lambda_a(K_a) \big)\subseteq \mu_a
(K'_{\beta(a)}), \]
\item a map $\gamma\co B \to J$,
\item for each $b\in B$, an embedding
\[ \varrho_b \co (K'_b, L'_b, \chi'_b) \to (W'_{\gamma(b)}, H'_{\gamma(b)}, \psi'_{\gamma(b)}) \]
and an embedding
\[ \sigma_b\co (K''_b, L''_b, \chi''_b) \to (W''_{\nu(\gamma(b))}, H''_{\nu(\gamma(b))}, \psi''_{\nu(\gamma(b))}) \]
such that 
\[ \tilde g_{\gamma(b)}\big( \varrho_b(K'_b) \big) \subseteq \sigma_b(K''_b).\]
\end{itemize}
Let $q\in Q$ and set $r\sceq f(q)$. We fix $i\in I$ and $j\in J$ such that $q\in \pi_i(V_i)$ and $r\in \psi'_j(W'_j)$. Further we choose $v'\in V'_{\tau(i)}$ and $w'\in W'_j$ such that $\pi'_{\tau(i)}(v') = r = \psi'_j(w')$. By compatibility of orbifold charts we find a restriction $(K'_q, L'_q, \chi'_q)$ of $(V'_{\tau(i)}, G'_{\tau(i)}, \pi'_{\tau(i)})$ with $v'\in K'_q$ and an embedding 
\[ \varrho_q\co (K'_q, L'_q, \chi'_q) \to (W'_j, H'_j, \psi'_j). \]
Since $\tilde f_i$ is continuous, there is a restriction $(K_q, L_q, \chi_q)$ of $(V_i, G_i, \pi_i)$ such that $q\in \chi_q(K_q)$ and $\tilde f_i(K_q) \subseteq K'_q$. We set 
\[
(K''_q, L''_q, \chi''_q) \sceq (W''_j, H''_j, \psi''_j).
\]
We consider orbifold charts constructed for distinct $q$ to be distinct. Then we set 
\begin{align*}
&A\sceq Q, \quad &\alpha(q) &\sceq i, \quad \lambda_q \sceq \id,  \quad \mu_q \sceq \id,
\\
&B \sceq Q \sqcup Q'\setminus f(Q), \quad & \beta(q) &\sceq q, \quad \gamma(q) \sceq j, \quad \sigma_q \sceq \id \quad\text{for $q\in Q$.} 
\end{align*}
For $q' \in Q'\setminus f(Q)$ we fix $j\in J$ with $q' \in \psi'_j(W'_j)$ and set $\gamma(q') \sceq j$. Further we set 
\[
(K'_{q'}, L'_{q'}, \chi'_{q'}) \sceq (W'_j, H'_j, \psi'_j)\quad\text{and}\quad (K''_{q'}, L''_{q'}, \chi''_{q'}) \sceq (W''_j, H''_j, \psi''_j).
\]
Again we consider orbifold charts build for distinct $q'$ to be distinct and to be distinct from all those defined for some $q\in Q$, and we define $\varrho_{q'} \sceq \id$ and $\sigma_{q'} \sceq \id$. Then all requirements are satisfied. If the orbifolds are second-countable, then we see as in Lemma~\ref{fortrans} that we may restrict this construction to countably many points of $Q$ and $Q'$ resp.\@ countable sets $A$ and $B$ to achieve that $\mc K, \mc K'$ and $\mc K''$ are countable.
\end{proof}

\begin{prop}\label{compositiongood}
 The composition in $\Orbmap$ is well-defined.
\end{prop}

\begin{proof} We use the notation from the definition of the composition. We have to show that the  composition of $[\hat f]$ and $[\hat g]$ neither depends on the choice of the induced orbifold maps $\hat h$ and $\hat k$ nor on the choice of the representatives of $[\hat f]$ and $[\hat g]$.
To prove independence of the choice of $\hat h$ and $\hat k$, suppose that we have two pairs $(\hat h_j, \hat k_j)$ of
induced orbifold maps $\hat h_j \in \Orbmap(\mathcal K_j, \mathcal K'_j)$, $\hat k_j \in
\Orbmap(\mathcal K'_j, \mathcal K''_j)$ ($j=1,2$) such that the diagram
\[
\xymatrix{
\mathcal K_1 \ar[rrr]^{\hat h_1} \ar[dr] &&& \mathcal K'_1 \ar[rrr]^{\hat k_1} \ar[dl] \ar[dr] &&& \mathcal K''_1 \ar[dl]
\\
& \mathcal V \ar[r]^{\hat f} & \mathcal V' && \mathcal W' \ar[r]^{\hat g} & \mathcal W''
\\
\mathcal K_2 \ar[rrr]^{\hat h_2} \ar[ur] &&& \mathcal K'_2 \ar[rrr]^{\hat k_2} \ar[ul] \ar[ur] &&& \mathcal K''_2 \ar[ul]
}
\]
commutes. The non-horizontal maps are lifts of identity. Lemma~\ref{fortrans}
shows the existence of representatives $\mathcal H$ of $\mathcal U$, $\mathcal
H', \mathcal I'$ of $\mathcal U'$, $\mathcal I''$ of $\mathcal U''$, and charted
orbifold maps $\hat h_3\in \Orbmap(\mathcal H, \mathcal H')$, $\hat k_3\in
\Orbmap(\mathcal I',\mathcal I'')$, and appropriate lifts of identity such that
the diagrams
\[
\xymatrix{
& \mathcal K_1 \ar[r]^{\hat h_1} & \mathcal K'_1
&&& \mathcal K'_1 \ar[r]^{\hat k_1} & \mathcal K''_1
\\
\mathcal H \ar[rrr]^{\hat h_3} \ar[ur] \ar[dr] &&& \mathcal H' \ar[ul] \ar[dl] 
& \mathcal I' \ar[rrr]^{\hat k_3} \ar[ur] \ar[dr] &&& \mathcal I'' \ar[ul] \ar[dl]
\\
& \mathcal K_2 \ar[r]^{\hat h_2} & \mathcal K'_2 
&&& \mathcal K'_2 \ar[r]^{\hat k_2} & \mathcal K''_2 
}
\]
commute. By Lemma~\ref{compwd} we find representatives $\mathcal K, \mathcal K',
\mathcal K''$ of $\mathcal U, \mathcal U', \mathcal U''$, resp.\@, charted
orbifold maps $\hat h\in \Orbmap(\mathcal K, \mathcal K')$, $\hat k\in
\Orbmap(\mathcal K',\mathcal K'')$, and appropriate lifts of identity such that 
\[
\xymatrix{
& \mathcal H \ar[r]^{\hat h_3} & \mathcal H' && \mathcal I' \ar[r]^{\hat k_3} & \mathcal I''
\\
\mathcal K \ar[rrr]^{\hat h} \ar[ur] &&& \mathcal K' \ar[rrr]^{\hat k} \ar[ur] \ar[ul]  &&& \mathcal K'' \ar[ul]
}
\]
commutes. Hence, altogether we have the commutative diagram
\begin{equation}\label{subtle}
\xymatrix{
\mathcal K_1 \ar[r]^{\hat h_1} & \mathcal K'_1 \ar[r]^{\hat k_1} & \mathcal K''_1
\\
\mathcal K \ar[r]^{\hat h} \ar[u] \ar[d] & \mathcal K' \ar[r]^{\hat k} \ar[u] \ar[d] & \mathcal K'' \ar[u] \ar[d]
\\
\mathcal K_2 \ar[r]^{\hat h_2} & \mathcal K'_2 \ar[r]^{\hat k_2} & \mathcal K''_2
}
\end{equation}
which shows that $\hat k_1 \circ \hat h_1$ and $\hat k_2\circ\hat h_2$ are equivalent (see Remark~\ref{details} below for some more details).

For the proof of the independence of the choices of the representatives of $[\hat f]$ and $[\hat g]$, let $\hat f_1 \in \Orbmap(\mathcal V_1,\mathcal V'_1)$, $\hat f_2 \in \Orbmap(\mathcal V_2,\mathcal V'_2)$ be representatives of $[\hat f]$, and $\hat g_1 \in \Orbmap(\mathcal W'_1, \mathcal W''_1)$, $\hat g_2\in \Orbmap(\mathcal W'_2,\mathcal W''_2)$ be representatives of $[\hat g]$. Further, for $j=1,2$, let $\hat h_j \in \Orbmap(\mathcal K_j, \mathcal K'_j)$ be induced by $\hat f_j$, and $\hat k_j \in \Orbmap(\mathcal K'_j, \mathcal K''_j)$ be induced by $\hat g_j$.
Since $\hat f_1$ and $\hat f_2$ are equivalent, we find representatives $\mathcal V$, $\mathcal V'$ of $\mathcal U$, $\mathcal U'$, resp.\@,  a charted orbifold map $\hat f\in \Orbmap(\mathcal V,\mathcal V')$ and appropriate lifts of identities, and analogously for $\hat g_1$ and $\hat g_2$, such that the diagrams
\[ 
\xymatrix{
& \mathcal V_1 \ar[r]^{\hat f_1} & \mathcal V'_1 
&&& \mathcal W'_1 \ar[r]^{\hat g_1} & \mathcal W''_1
\\
\mathcal V \ar[rrr]^{\hat f} \ar[ur] \ar[dr] &&& \mathcal V' \ar[ul] \ar[dl]
& \mathcal W' \ar[rrr]^{\hat g} \ar[ur] \ar[dr] &&& \mathcal W'' \ar[ul] \ar[dl]
\\
& \mathcal V_2 \ar[r]^{\hat f_2} & \mathcal V'_2
&&& \mathcal W'_2 \ar[r]^{\hat g_2} & \mathcal W''_2
}
\]
commute. Lemma~\ref{compwd} yields the existence of $\hat h\in \Orbmap(\mathcal K, \mathcal K')$ and $\hat k\in \Orbmap(\mathcal K', \mathcal K'')$ and appropriate lifts of identities such that 
\[
\xymatrix{
& \mathcal V \ar[r]^{\hat f} & \mathcal V' && \mathcal W' \ar[r]^{\hat g} & \mathcal W''
\\
\mathcal K \ar[rrr]^{\hat h} \ar[ur] &&& \mathcal K' \ar[rrr]^{\hat k} \ar[ul] \ar[ur] &&& \mathcal K'' \ar[ul]
}
\]
commutes. Since $\hat h$ is induced by $\hat f_1$ and by $\hat f_2$, and
likewise, $\hat k$ is induced by $\hat g_1$ and by $\hat g_2$, we conclude as above that $\hat k_1 \circ \hat h_1$ and $\hat k_2\circ \hat
h_2$ are both equivalent to $\hat k\circ\hat h$. This yields that
the composition map is well-defined.
\end{proof}

\begin{remark}\label{details}
We elaborate on a subtle issue in the proof of Proposition~\ref{compositiongood}. By the reasoning in the proof, the diagram \eqref{subtle} \textit{a priori} is of the form
\[
\xymatrix{
\mathcal K_1 \ar[r]^{\hat h_1} & \mathcal K'_1 && \mathcal K'_1 \ar[r]^{\hat k_1} & \mathcal K''_1
\\
\mc H \ar[r]^{\hat h_3}\ar[u] & \mc H'\ar[u] && \mc I'\ar[r]^{\hat k_3}\ar[u] & \mc I''\ar[u]
\\
\mathcal K \ar[rr]^{\hat h} \ar[u]\ar[d] && \mathcal K' \ar[rr]^{\hat k} \ar[ur]\ar[dr] \ar[ul]\ar[dl] && \mathcal K'' \ar[u]\ar[d]
\\
\mc H \ar[r]^{\hat h_3}\ar[d] & \mc H'\ar[d] && \mc I'\ar[r]^{\hat k_3}\ar[d] & \mc I''\ar[d]
\\
\mathcal K_2 \ar[r]^{\hat h_2} & \mathcal K'_2 && \mathcal K'_2 \ar[r]^{\hat k_2} & \mathcal K''_2
}
\]
and we implicitly claim that the two maps from $\mc K'$ to $\mc K'_1$ as well as from $\mc K'$ to $\mc K'_2$  are identical. This might lead to the suspicion that the composition of charted orbifold maps as defined above is not well-defined.

\begin{figure}[h]
\[
\xymatrix{
K \ar[d] \ar[rr]^{h} && K'\subseteq K'_1 \ar[dl]^{\eps_1} \ar[dr]_{\eps_2} \ar[rr]^{k} && K'' \ar[d]
\\
H \ar[r]^{h_3 = f\vert_H} \ar[d] & H' \ar[dr]^{\eps_1^{-1}} && I' \ar[dl]_{\eps_2^{-1}} \ar[r]^{k_3 = g\vert_{I'}} & I'' \ar[d]
\\
K_1 \ar[rr]^{h_1} \ar[d] && K'_1 \ar[rr]^{k_1} \ar[dl]^{\eps_1} \ar[dr]_{\eps_2} && K''_1 \ar[d]
\\
V \ar[r]^{f} & V' && W' \ar[r]^{g} & W''
\\
K_2 \ar[u] \ar[rr]^{h_2} && K'_2 \ar[ul]_{\eta_1} \ar[ur]^{\eta_2} \ar[rr]^{k_2} && K''_2 \ar[u]
\\
H \ar[u] \ar[r]^{h_3 = f\vert_H} & H' \ar[ur]_{\eta_1^{-1}} && I' \ar[ul]^{\eta_2^{-1}} \ar[r]^{k_3 = g\vert_{I'}} & I'' \ar[u]
\\
K \ar[u] \ar[rr]^{h} && K' \ar[ul]_{\xi^{-1}\circ\eta_1} \ar[ur]^{\xi^{-1}\circ\eta_2} \ar[rr]^{k} && K'' \ar[u]
}
\]
\caption{Maps in local charts}\label{fig:subtle}
\end{figure}

However, since $\hat h, \hat h_1, \hat h_2$ are all induced by $\hat f$, and $\hat k, \hat k_1, \hat k_2$ are all induced by $\hat g$, the arising maps $\mc K'\to \mc K'_1$ are indeed identical as well as the arising maps $\mc K'\to \mc K'_2$. The diagram in Figure~\ref{fig:subtle} shows the maps between the local charts, which proves the latter statement as well as that diagram~\ref{subtle} is indeed commutative. 
The notation in this diagram is intuitive, e.g. $K$ is an orbifold chart in $\mc K$, all greek letters refer to diffeomorphisms. To improve readibility we may assume without loss of generality that $K'\subseteq K'_1$, $K'_1 \subseteq V'\cap W'$ and $K'_2 \subseteq V'\cap W'$. The latter two choices implicitly mean that we assume that the local diffeomorphism between $V'$ and $W'$ is the identity, which we indeed may. Otherwise we would just need to add the same local diffeomorphism to several maps in the diagram. The local diffeomorphism between $K'$ and $K'_2$ is denoted by $\xi$ (note that we are not allowed to suppose that $\xi=\id$ in general). Note that the top and bottom maps in the diagram are identical.

\end{remark}

We end this section with a discussion of the equivalence class represented by a lift of identity. The following proposition shows that it is precisely the class of all lifts of identity of the considered orbifold. This justifies the notion ``identity morphism'' in Definition~\ref{liftiddef}.

\begin{prop}\label{equivclassid}
Let $(Q,\mathcal U)$ be an orbifold and $\eps$ a lift of $\id_{(Q,\mc U)}$. Then the equivalence class $[\eps]$ of $\eps$ consists precisely of all lifts of $\id_{(Q,\mc U)}$.
\end{prop}

\begin{proof} Let $\eps_1 \in \Orbmap(\mathcal V_1, \mathcal W_1)$ and $\eps_2
\in \Orbmap(\mathcal V_2, \mathcal W_2)$ be two lifts of $\id_{(Q,\mc U)}$.
Propositions~\ref{induceslifts}  and \ref{extending} imply that there is a representative $\mc V$ of $\mc U$ such that $\eps_1$ and
$\eps_2$ both induce the orbifold map 
\[ \widehat \id_Q \sceq ( \id_Q, \{ \id_{V_i} \}_{i\in I}, [R,\sigma] ) \]
with $(\mathcal V,\mc V)$. Thus, each two lifts of $\id_{(Q,\mc U)}$ are equivalent. 

Let $\hat f$ be a charted orbifold map which is equivalent to $\eps$.
W.l.o.g.\@ we may assume that $\eps=\widehat\id_Q$. To fix notation let 
\begin{align*}
\mathcal V & = \{ (V_i, G_i, \pi_i) \mid i\in I\}, \text{indexed by $I$,}
\\
\mathcal K_1 & = \{ (K_{1,a}, L_{1,a}, \chi_{1,a} ) \mid a\in A \}, \text{ indexed by $A$,}
\\
\mathcal K_2 & = \{ (K_{2,b}, L_{2,b}, \chi_{2,b} ) \mid b\in B \}, \text{ indexed by $B$,}
\\
\mathcal W_1 & = \{ (W_{1,j}, H_{1,j}, \psi_{1,j} ) \mid j\in J \}, \text{ indexed by $J$,}
\\
\mathcal W_2 & = \{ (W_{2,k}, H_{2,k}, \psi_{2,k} ) \mid k\in K \}, \text{ indexed by $K$,} 
\end{align*}
be representatives of $\mathcal U$. Let 
\[
 \hat f  = \big( f, \{ \tilde f_j\}_{j\in J}, [P_f,\nu_f] \big) \in \Orbmap(\mc W_1,\mc W_2).
\]
Suppose that 
\[
\hat g = \big( g, \{ \tilde g_a\}_{a\in A}, [P_g,\nu_g] \big) \in \Orbmap(\mc K_1,\mc K_2)
\]
is a charted orbifold map and
\begin{align*}
\eps_1 & = \big( \id_Q, \{ \lambda_{1,a} \}_{a\in A}, [P_1,\nu_1] \big) \in \Orbmap(\mc K_1,\mc V)
\\
\eps_2 & = \big( \id_Q, \{ \lambda_{2,a} \}_{a\in A}, [P_2,\nu_2] \big) \in \Orbmap(\mc K_1,\mc W_1)
\\
\delta_1 & = \big( \id_Q, \{ \mu_{1,b} \}_{b\in B}, [R_1,\sigma_1] \big) \in  \Orbmap(\mc K_2, \mc V)
\\
\delta_2 & = \big( \id_Q, \{ \mu_{2,b} \}_{b\in B}, [R_2, \sigma_2] \big) \in \Orbmap(\mc K_2,\mc W_2)
\end{align*}
are lifts of $\id_{(Q,\mc U)}$ such that the diagram (which shows that $\hat f$ and $\widehat \id_Q$ are equivalent)
\[
\xymatrix{
& \mathcal V \ar[r]^{\widehat \id_Q} & \mathcal V
\\
\mathcal K_1 \ar[ur]^{\eps_1} \ar[dr]_{\eps_2} \ar[rrr]^{\hat g} &&& \mathcal K_2 \ar[ul]_{\delta_1} \ar[dl]^{\delta_2}
\\
& \mathcal W_1 \ar[r]^{\hat f} & \mathcal W_2
}
\]
commutes. Clearly, $g=\id_Q$ and hence $f=\id_Q$. Let $\alpha \co A\to I$, $\beta\co A \to J$, $\gamma\co A\to B$, $\delta \co B \to I$, $\eta\co B \to K$ and $\zeta\co J \to K$ be the induced maps on the index sets as, e.g., in Construction~\ref{constr_comp}. For each $a\in A$, we
have 
\[ \id_{V_{\alpha(a)}}\circ \lambda_{1,a} = \mu_{1,\gamma(a)}\circ \tilde g_a.\]
Since $\id_{V_{\alpha(a)}}$, $\lambda_{1,a}$ and $\mu_{1,\gamma(a)}$ are local diffeomorphisms, so is $\tilde g_a$. Now 
\[ \tilde f_{\beta(a)} \circ \lambda_{2,a} = \mu_{2,\gamma(a)} \circ \tilde g_a\]
for each $a\in A$. Hence $\tilde f_{\beta(a)}$ is a local diffeomorphism.
Lemma~\ref{welldefinedll} implies that $\tilde f_j$ is a local diffeomorphism
for each $j\in J$. Therefore, $\hat f$ is a lift of $\id_{(Q,\mc U)}$. 
\end{proof}

\section{The orbifold category in terms of marked atlas groupoids}\label{atlascategory}

Proposition~\ref{conclusion} and Remark~\ref{F1_equivariant} show that charted orbifold maps and their composition correspond to homomorphisms between marked atlas grou\-poids and their composition. By characterizing lifts of identity and equivalence of charted orbifold maps in terms of marked atlas grou\-poids and their homomorphisms, we construct a category for marked atlas groupoids which is isomorphic to the one of reduced orbifolds. To that end we first show that lifts of identity correspond to unit weak equivalences, a notion we define below. Throughout this section let $\pr_1$ denote the projection onto the first component.

A homomorphism $\varphi=(\varphi_0,\varphi_1)\co G\to H$ between Lie groupoids is called a \textit{weak equivalence} if
\begin{enumerate}[(i)]
\item  the map
\[
 t\circ \pr_1\co H_1\pullback{s}{\varphi_0} G_0 \to H_0
\]
is a surjective submersion, and
\item the diagram
\[
\xymatrix{
G_1 \ar[r]^{\varphi_1}\ar[d]_{(s,t)} & H_1 \ar[d]^{(s,t)}
\\
G_0\times G_0 \ar[r]^{\varphi_0\times\varphi_0} & H_0\times H_0
}
\]
is a fibered product.
\end{enumerate}
Two Lie groupoids $G,H$ are called \textit{Morita equivalent} if there is a Lie
groupoid $K$ and weak equivalences 
\[
\xymatrix{
G & K \ar[l]_{\varphi} \ar[r]^{\psi} & H.
}
\]

\begin{defi}
Let $(G_1,\alpha_1,X_1)$ and $(G_2,\alpha_2,X_2)$ be marked atlas groupoids. A homomorphism
\[
\varphi=(\varphi_0,\varphi_1)\co (G_1,\alpha_1, X_1) \to (G_2,\alpha_2, X_2)
\]
is called a \textit{unit weak
equivalence} if $\varphi\co G_1 \to G_2$ is
a weak equivalence and $\alpha_2\circ |\varphi|
\circ\alpha_1^{-1} = \id_{X_1}$. Necessarily we have $X_1=X_2\seqc X$.
A \textit{unit Morita equivalence} between $(G_1,\alpha_1, X)$ and $(G_2,\alpha_2, X)$ is a pair $(\psi_1, \psi_2)$ of unit weak equivalences 
\[
\psi_j\co (G,\alpha, X) \to (G_j,\alpha_j, X)
\]
where $(G,\alpha, X)$ is some marked atlas groupoid. If such a unit Morita equivalence exists, then the marked atlas groupoids $(G_1,\alpha_1, X)$ and $(G_2,\alpha_2, X)$ are called \textit{unit Morita equivalent}.  
\end{defi}

In contrast to Morita equivalence of Lie groupoids, unit Morita equivalence of marked atlas groupoids requires the third (marked) Lie groupoid to be an atlas groupoid. In Proposition~\ref{Morequivgr} below we will show that unit Morita equivalence of marked atlas groupoids is indeed an equivalence relation. 

The following proposition identifies lifts of identity with unit weak equivalences. 

\begin{prop}\label{charuwe}
Let $\mathcal U$ and $\mathcal U'$ be orbifold structures on the topological space $Q$. Further let $\mc V$ resp.\@ $\mc W'$ be a representative of $\mathcal U$ resp.\@ of  $\mathcal U'$.
\begin{enumerate}[{\rm (i)}]
\item\label{charuwei} Suppose that $\mc U = \mc U'$. If $\hat f\in \Orbmap(\mathcal V,\mathcal W')$ is a lift
of $\id_{(Q,\mc U)}$, then $F_1(\hat f)$  is a unit weak equivalence.
\item\label{charuweii} Let $\eps \in \Hom(\Gamma(\mathcal V), \Gamma(\mathcal
W'))$ be a unit weak equivalence. Then $\mc U = \mc U'$, and $F_2(\eps)$ is a
lift of $\id_{(Q,\mc U)}$.  
\end{enumerate}
\end{prop}

\begin{proof} Let 
\[ \mathcal V = \{ (V_i, G_i, \pi_i) \mid i\in I\}\quad \text{resp.\@}\quad \mathcal W' = \{ (W'_j, H'_j, \psi'_j)\mid j\in J\},\]
indexed by $I$ resp.\@ by $J$, and let $G\sceq \Gamma(\mathcal V)$ and $H\sceq \Gamma(\mathcal W')$.
We will first prove \eqref{charuwei}. Suppose $\hat f = (\id_Q, \{
\tilde f_i \}_{i\in I}, [P,\nu])$. By Proposition~\ref{orb_gr} it suffices to
show that $\eps=(\eps_0,\eps_1)\sceq F_1(\hat f)$ is a weak equivalence. We
first show that 
\[ 
t\circ \pr_1 \co 
\left\{
\begin{array}{ccc}H_1 \pullback{s}{\eps_0} G_0 & \to & H_0
\\
(h,x) & \mapsto & t(h)
\end{array}
\right.
\]
is a submersion. Let $(h,x)\in H_1\pullback{s}{\eps_0} G_0$. Recall from
Proposition~\ref{induceslifts} that $\eps_0$ is a local diffeomorphism, and from
Special Case~\ref{atlasgroupoid} that $G$ and $H$ are \'etale groupoids. Choose
open neighborhoods $U_x$ of $x$ in $G_0$ and $U_h$ of $h$ in $H_1$ such that
$\eps_0\vert_{U_x}$ and $s\vert_{U_h}$ are embeddings with $s(U_h) =
\eps_0(U_x)$. Then $U_h\pullback{s}{\eps_0}U_x$ is open in
$H_1\pullback{s}{\eps_0}G_0$. Further
\begin{align*}
U_h\pullback{s}{\eps_0}U_x & = \left\{ (k,y)\in U_h\times U_x \left\vert\  s(k) = \eps_0(y) \right.\right\}
\\ & = \left\{ \left(k, \eps_0^{-1}(s(k))\right) \left\vert\ k\in U_h \vphantom{\left(k, \eps_0^{-1}(s(k))\right) } \right.\right\}.
\end{align*}
Therefore, 
\[ \pr_1 \co U_h\pullback{s}{\eps_0}U_x \to U_h \]
is a diffeomorphism. Since $t$ is a local diffeomorphism, $t\circ\pr_1$ is a submersion. 

Now we prove that $t\circ\pr_1$ is surjective. Let $y\in H_0$, say $y\in W'_j$, and set $\psi'_j(y) \seqc q\in Q$. Then there is an orbifold chart $(V_i, G_i, \pi_i)\in \mathcal V$ such that $q\in \pi_i(V_i)$, say $q=\pi_i(x)$. 
\[
\xymatrix{
V_i \ar[rr]^{\tilde f_i}\ar[dr]_{\pi_i} && W'_i \ar[dl]^{\psi'_i}
\\
& Q 
}
\]
Set $z\sceq \tilde f_i(x)$, hence $\psi'_i(z) = q = \psi'_j(y)$. Hence, there are a restriction $(S',K',\chi')$ of  $(W'_i,H'_i,\psi'_i)$ with $z\in S'$ and an embedding 
\[ \lambda\co (S', K', \chi') \to (W'_j, H'_j, \psi'_j)\]
such that $\lambda(z) = y$. Then $\lambda\in \Psi(\mathcal W')$ and $(\germ_z \lambda, x)\in H_1 \pullback{s}{\eps_0} G_0$ with 
\[
t\circ \pr_1(\germ_z\lambda, x) = t(\germ_z\lambda) = y.
\]
This means that $t\circ\pr_1$ is surjective. 

Set
\[ K \sceq (G_0\times G_0) \pullback{(\eps_0,\eps_0)}{(s,t)} H_1.\]
It remains to show that the map
\[ \beta\co\left\{
\begin{array}{ccc}
G_1 & \to & K
\\
\germ_x g & \mapsto & (x, g(x), \eps_1(\germ_x g))
\end{array}\right.
\]
is a diffeomorphism. Note that $\beta = (s, t, \eps_1)$.
Let $(x,y,\germ_{\eps_0(x)}h)$ be in $K$, hence $\germ_{\eps_0(x)} h\co \eps_0(x) \to \eps_0(y)$. By the definition of $H_1$ there are open neighborhoods $U'_1$ of $\eps_0(x)$ and $U'_2$ of $\eps_0(y)$ in $W'\sceq \coprod_{j\in J} W'_j$ such that $h\co U'_1 \to U'_2$ is an element of $\Psi(\mathcal W')$. Since $\eps_0$ is a local diffeomorphism, there are open neighborhoods $U_1$ of $x$ and $U_2$ of $y$ in $V\sceq \coprod_{i\in I} V_i$ such that $\eps_0\vert_{U_k}$ is an embedding with $\eps_0(U_k) \subseteq U'_k$ ($k=1,2$). After shrinking $U'_k$ we can assume that $\eps_0(U_k) = U'_k$. Let $\gamma_k\sceq \eps_0\vert_{U_k}$. Then
\[ g\sceq \gamma_2^{-1}\circ h \circ \gamma_1 \co U_1 \to U_2 \]
is a diffeomorphism, hence $g\in \Psi(\mathcal V)$. By Proposition~\ref{extending} (or rather its proof), we have $\nu(g) = h$ and hence (see (\apref{R}{invariant}{}) and recall Proposition~\ref{orb_gr})
\[
\eps_1(\germ_x g) = \germ_{\eps_0(x)} h.
\]
Finally, we see 
\[ \beta(\germ_x g) = (x, g(x), \eps_1(\germ_x g)) = (x,y, \germ_{\eps_0(x)} h). \] 
Therefore $\beta$ is surjective. Since $\germ_x g$ does not depend on the choice of $U_k$ and $U'_k$, the map $\beta$ is also injective. Finally, we will show that $\beta$ is a local diffeomorphism. Since $s$ and $t$ are local diffeomorphisms, we only need to prove that $\eps_1$ is one as well. Let $\germ_x f\in G_1$. Choose an open neighborhood $U$ of $x$ such that $U\subseteq \dom f$ and $\eps_0\vert_U\co U\to \eps_0(U)$ is a diffeomorphism. By the germ topology, the set
\[ \tilde U \sceq \{ \germ_y f\mid y\in U \}\]
is open in $G_1$, and the set 
\[ \tilde V \sceq \{ \germ_z \nu(f) \mid z\in \eps_0(U)\} \]
is open in $H_1$. Further the diagrams
\[
\xymatrix{
\tilde U \ar[r]^{\eps_1} \ar[d] & \tilde V\ar[d] &&& \germ_yf \ar@{|->}[r]^{\eps_1} \ar@{|->}[d] & \germ_{\eps_0(y)} \nu(f)\ar@{|->}[d] 
\\
U \ar[r]_{\eps_0} & \eps_0(U) &&& y \ar@{|->}[r]_{\eps_0} & \eps_0(y)
}
\]
commute. Since the vertical arrows are diffeomorphisms by definition, \[\eps_1\vert_{\tilde U} \co \tilde U \to \tilde V\] is a diffeomorphism as well. This
completes the proof of \eqref{charuwei}.

We will now prove \eqref{charuweii}. Proposition~\ref{backorb} shows that the
orbifold atlases $\mc V$ and $\mc W'$ are determined completely by the marked
atlas groupoids $\Gamma(\mc V)$ and $\Gamma(\mc W')$, resp. Hence we can apply
Proposition~\ref{inducedorbgr}, which shows that $F_2(\eps)$ is well-defined. Suppose
that
\[
 F_2(\eps) = \big( f, \{\tilde f_i\}_{i\in I}, [P,\nu]).
\]
Proposition~\ref{inducedorbgr} yields $f=\id_Q$. By  \cite[Exercises~5.16(4)]{Moerdijk_Mrcun} $\eps_0$
is a local diffeomorphism. Thus, Proposition~\ref{inducedorbgr} implies that each
$\tilde f_i$ is a local diffeomorphism. The domain atlas of $F_2(\eps)$ is $\mc
V$, its range family is $\mc W'$. From Proposition~\ref{samestructure} it follows
that $\mc U = \mc U'$. By Definition~\ref{liftiddef} $F_2(\eps)$ is a lift of $\id_{(Q,\mc U)}$.
\end{proof}

The combination of Propositions~\ref{backorb}, ~\ref{charuwe} and Remark~\ref{F1_equivariant} now allows to identity  each step in  the construction of the category of reduced orbifolds  and each intermediate object in terms of marked atlas groupoids. For an orbifold $(Q,\mc U)$ we define
\[
 \Gamma(Q,\mc U) \sceq \big\{ \big(\Gamma(\mc V), \alpha_{\mc V}, Q \big) \
\big\vert\ \text{$\mc V$ is a (countable) representative of $\mc U$} \big\}.
\]
We recall from Special Case~\ref{atlasgroupoid} the convention that for second-countable orbifolds we always restrict here to countable representatives of the orbifold structure. For non-second-countable, we also consider non-countable representatives.

Then Proposition~\ref{equivclassid} gives rise to the following proposition.

\begin{prop}\label{Morequivgr}
Unit Morita equivalence of marked atlas groupoids is an
equivalence relation. Further, if $\mc V$ is a (countable) representative of the orbifold structure $\mc U$ of the orbifold $(Q,\mc U)$, then the unit Morita equivalence class of $(\Gamma(\mc V), \alpha_{\mc V}, Q)$ is $\Gamma(Q,\mc U)$.
\end{prop}

Equivalence of charted orbifold maps translates to marked atlas groupoids as follows.

\begin{defi}
Let $(G_1,\alpha_1,X)$, $(G_2,\alpha_2,X)$, as well as $(H_1,\beta_1,Y)$ and $(H_2,\beta_2,Y)$ be marked atlas groupoids. For $j=1,2$ let 
\[
\psi_j \co (G_j,\alpha_j, X) \to (H_j,\beta_j, Y)
\]
be a homomorphism of marked Lie groupoids.  We call $\psi_1$ and $\psi_2$ \textit{unit Morita equivalent} if there exist marked atlas groupoids $(G,\alpha,X)$ and $(H,\beta,Y)$, a homomorphism $\chi\co (G,\alpha, X) \to (H,\beta, Y)$, and unit weak equivalences $\eps_j\co (G,\alpha, X) \to (G_j,\alpha_j, X)$, $\delta_j\co (H,\beta, Y) \to (H_j,\beta_j, Y)$ such that the diagram
\[
\xymatrix{
& (G_1,\alpha_1,X) \ar[r]^{\psi_1} & (H_1,\beta_1,Y)
\\
(G,\alpha, X) \ar[ur]^{\eps_1} \ar[dr]_{\eps_2} \ar[rrr]^{\chi} &&& (H,\beta, Y) \ar[ul]_{\delta_1} \ar[dl]^{\delta_2}
\\
& (G_2,\alpha_2,X) \ar[r]^{\psi_2} & (H_2,\beta_2,Y)
}
\]
commutes.
\end{defi}

Proposition~\ref{mapwell} in terms of atlas groupoids means the following.

\begin{prop}
Unit Morita equivalence of homomorphisms between marked atlas groupoids is an
equivalence relation.
\end{prop}

We define the category $\Agr$ of marked atlas groupoids as follows: Its class of
objects consists of all $\Gamma(Q,\mathcal U)$. The morphisms from $\Gamma(Q,\mc
U)$ to $\Gamma(Q',\mc U')$ are the unit Morita equivalence classes $[\varphi]$
of homomorphisms $\varphi\co (G,\alpha, Q) \to
(G',\alpha', Q')$ where $(G,\alpha, Q)$ is any representative of
$\Gamma(Q,\mc U)$ and $(G',\alpha',Q')$ is any representative of
$\Gamma(Q',\mc U')$. 

The composition of two morphisms $[\varphi]\in \Morph\big(\Gamma(Q,\mc U),\Gamma(Q',\mc U')\big)$ and $[\psi]\in\Morph\big(\Gamma(Q',\mc U'),\Gamma(Q'',\mc U'')\big)$ is defined as follows:  Choose representatives 
\[
\varphi\co (G,\alpha,Q)\to (G',\alpha',Q')\quad\text{of}\ [\varphi]
\]
and 
\[
\psi\co(H',\beta',Q')\to (H'',\beta'',Q'')\quad \text{of}\ [\psi].
\]
Then find representatives $(K,\gamma,Q)$, $(K',\gamma',Q')$, $(K'',\gamma'',Q'')$ of the classes $\Gamma(Q,\mc U)$, $\Gamma(Q',\mc U')$, $\Gamma(Q'',\mc U'')$, resp., and unit Morita equivalences
\begin{align*}
\eps & \co (K,\gamma, Q) \to (G,\alpha, Q),
\\
\eps'_1 & \co (K',\gamma', Q') \to (G',\alpha',Q'),
\\
\eps'_2 & \co (K',\gamma',Q') \to (H',\beta',Q'),
\\
\eps'' & \co (K'',\gamma'',Q'') \to (H'',\beta'', Q''),
\end{align*}
and homomorphisms of marked Lie groupoids
\begin{align*}
 \chi & \co (K,\gamma,Q) \to (K',\gamma', Q'),
\\
\kappa & \co (K',\gamma', Q') \to (K'',\gamma'', Q'')
\end{align*}
such that the diagram
\[\def\objectstyle{\scriptstyle}
\xymatrix{
 (G,\alpha,Q) \ar[r]^\varphi & (G',\alpha',Q') & & (H',\beta',Q') \ar[r]^\psi & (H'',\beta'', Q'')
\\
(K,\gamma,Q) \ar[u]^\eps \ar[rr]^\chi && (K',\gamma', Q') \ar[ul]_{\eps'_1} \ar[ur]^{\eps'_2} \ar[rr]^\kappa && (K'',\gamma'', Q'') \ar[u]_{\eps''}
}
\]
commutes. Then the composition of $[\varphi]$ and $[\psi]$ is defined as
\[
 [\psi] \circ [\varphi] \sceq [\kappa\circ\chi].
\]

Invoking  Lemmas~\ref{onlyinduced}, \ref{compwd} and Proposition~\ref{compositiongood} we deduce the following proposition.

\begin{prop}
The composition in $\Agr$ is well-defined.
\end{prop}

We define an
assignment $F$ from the orbifold category $\Orbmap$ to the category of marked atlas groupoids $\Agr$ as follows. On the level of objects, $F$
maps the orbifold $(Q,\mc U)$ to $\Gamma(Q,\mc U)$. Suppose that $[\hat f]$ is
a morphism from the orbifold $(Q,\mc U)$ to the orbifold $(Q',\mc U')$. Then $F$
maps $[\hat f]$ to the morphism $[F_1(\hat f)]$ from $\Gamma(Q,\mc U)$ to
$\Gamma(Q',\mc U')$.

\begin{thm}
The assignment $F$ is a covariant functor from $\Orbmap$ to $\Agr$. Even more,
$F$ is an isomorphism of categories. The functor $F$ and its inverse are constructive. 
\end{thm}

In the following example we show that the representatives of orbifold maps from Example~\ref{Pfinite} define different orbifold maps. In this example we use $G(x,y)$ to denote the set of arrows $g$ of the groupoid $G$ with $s(g)=x$ and $t(g) =y$.

\begin{example}
Recall the representatives of orbifold maps 
\[
\hat f_1 = (f,\wt f, P, \nu_1)\quad\text{and}\quad \hat f_2 = (f, \wt f, P, \nu_2)
\]
from Example~\ref{Pfinite} and \ref{differenthomoms}. We claim that $\hat f_1$ and $\hat f_2$ are representatives of different orbifold maps. Assume for contradiction that $\hat f_1$ and $\hat f_2$ define the same orbifold map on $(Q,\mc U_1)$. This means that the groupoid homomorphisms $\varphi$ and $\psi$ from Example~\ref{differenthomoms} are Morita equivalent. Hence there exist marked atlas groupoids $K$ and $H$, unit weak equivalences 
\begin{align*}
 \alpha = (\alpha_0,\alpha_1) & \co K \to \Gamma,
&
\gamma = (\gamma_0,\gamma_1) & \co H \to \Gamma,
\\
\beta = (\beta_0,\beta_1) & \co K \to \Gamma,
&
\delta = (\delta_0,\delta_1) & \co H \to \Gamma,
\end{align*}
and a homomorphism $\chi= (\chi_0,\chi_1) \co K \to H$ such that the diagram
\[
\xymatrix{
& \Gamma \ar[r]^\varphi & \Gamma
\\
K \ar[ur]^{\alpha} \ar[dr]_\beta \ar[rrr]^\chi & & & H \ar[ul]_\gamma \ar[dl]^\delta
\\
& \Gamma \ar[r]^\psi & \Gamma
}
\]
commutes. Since $\alpha$ is a (unit) weak equivalence, there exists $x\in K$ and $g\in\Gamma$ with $s(g) = \alpha_0(x)$ and $t(g) = 0$. Necessarily, $g\in \{ \germ_0 (\pm\id) \}$, and hence $\alpha_0(x) = 0$. In turn, $\alpha_1$ induces a bijection between $K(x,x)$ and $\Gamma(0,0)$. Thus $K(x,x)$ consists of two elements, say $K(x,x) = \{ k_1,k_2 \}$. Let $x' \sceq \chi_0(x)$. Then 
\[
 0 = \varphi_0( \alpha_0(x) ) = \gamma_0( x' ).
\]
This shows that $\gamma_1$ induces a bijection between $H(x',x')$ and $\Gamma(0,0)$. For $j=1,2$ we have
\[
 \gamma_1(\chi_1(k_j)) = \varphi_1(\alpha_1(k_j)) = \germ_0 \id,
\]
which implies that $\chi_1(k_1) = \chi_1(k_2)$. Further $\beta_1$ induces a bijection between $K(x,x)$ and $\Gamma(\beta_0(x),\beta_0(x))$. 
Hence $\beta_0(x) = 0$, and thus
\[
 \psi_1(\beta_1(k_1)) \not= \psi_1(\beta_1(k_2)).
\]
But this contradicts to
\[
 \psi_1(\beta_1(k_1)) =  \delta_1(\chi_1(k_1)) = \delta_1(\chi_1(k_2))= \psi_1(\beta_1(k_2)).
\]
In turn, $\varphi$ and $\psi$ are not Morita equivalent.
\end{example}

\section{Extension to marked proper effective \'etale Lie groupoids}\label{extension}

Let $\Pgr$ denote the category of marked proper effective \'etale Lie groupoids which is defined analogously to the category $\Agr$ of marked atlas groupoids with the only difference that any marked atlas groupoid may be a marked proper effective \'etale Lie groupoid.
Note that the objects in $\Pgr$ are not the single marked proper effective \'etale Lie groupoids themselves but the unit Morita equivalence classes of these.

In this final section we show that these two categories are isomorphic via the canonical embedding. As before we consider two different flavors of categories distinguished by whether the involved topological spaces are required to be just paracompact or even second-countable. These differences do not cause any distinctions in statements or proofs and therefore will be implicit.

\begin{thm}
The embedding functor $\Agr \to \Pgr$ is an isomorphism.
\end{thm}

\begin{proof}
We prove that we can lift any given homomorphism (weak equivalence) between two proper effective \'etale Lie groupoids to a homomorphism (weak equivalence) between weakly equivalent atlas groupoids:
\[
\xymatrix{
& G \ar[r] & G'
\\
\Gamma(\mc U) \ar@{-->}[ur] \ar@{-->}[rrr] &&& \Gamma(\mc U'). \ar@{-->}[ul]
}
\]
Let $G$ be a proper effective \'etale Lie groupoid. By \cite[Corollary 5.31]{Moerdijk_Mrcun} there exists an orbifold atlas $\mc U$ on the base space $G_0$ of $G$ such that $\Gamma(\mc U)$ is weakly equivalent to $G$. More precisely, let $(U_i)_{i\in I}$ be a covering of $G_0$ by open connected subsets such that for each $i\in I$, there exists a point $x\in U_i$ such that 
\[
 G_i \sceq G_x \sceq \{ g\in G_1 \mid \stackrel{g}{x\longrightarrow x} \} = G_1\vert_{U_i}
\]
and $U_i$ is $G_i$-stable, and such that the action groupoid $G_i \ltimes U_i$ is isomorphic to $G\vert_{U_i}$. Then the family
\[
 \mc U \sceq \{ (U_i, G_i, \pr) \mid i\in I\}
\]
is an orbifold atlas on $G_0$. We identify $U_i$ with $\{i\}\times U_i$. Let $H\sceq \Gamma(\mc U)$. Then the weak equivalence $\varphi\colon H\to G$ is given by
\begin{align*}
 \varphi_0&\colon H_0 \to G_0,\quad (i,x) \mapsto x
\intertext{and}
\varphi_1 &\colon H_1 \to G_1,\quad \stackrel{k\quad}{(i,x)\longrightarrow (j,y)\quad } \mapsto \stackrel{\quad k}{\quad x\longrightarrow y}.
\end{align*}
Now let $G, G'$ be two proper effective \'etale Lie groupoids and $\chi\colon G\to G'$ be a homomorphism (weak equivalence). Choose orbifold atlases (in the way as above)
\[
 \mc U \sceq \{ (U_i, G_i, \pr) \mid i\in I \}, \quad \mc U' \sceq \{ (U'_j, G'_j, \pr) \mid j\in J \}
\]
on $G_0$ resp.\@ $G'_0$ such that there exists a map $\beta\colon I\to J$ such that 
\[
 \chi_0(U_i) \subseteq U'_{\beta(i)}
\]
for all $i\in I$ (this is possible). Set $H\sceq \Gamma(\mc U)$, $H'\sceq \Gamma(\mc U')$ and define $\alpha\colon H \to H'$ via
\begin{align*}
 \alpha_0 & \colon H_0 \to H'_0,\quad (i,x) \mapsto (\beta(i), \chi_0(x))
\\
\alpha_1 & \colon H_1 \to H'_1,\quad \stackrel{k\quad}{(i,x) \longrightarrow (j,y)\quad} \mapsto \stackrel{\quad\chi_1(k)}{\quad(\beta(i),\chi_0(x)) \longrightarrow (\beta(j),\chi_0(y))}.
\end{align*}
One easily checks that $\alpha$ is a homomorphism (weak equivalence). 

With this lifting property we can bring back all equivalence relations between marked proper effective \'etale Lie groupoids and between maps between these groupoids to statements purely between marked atlas groupoids and maps between those. This proves the theorem.
\end{proof}

\bibliographystyle{amsplain}

\begin{thebibliography}{10}

\bibitem{Adem_Leida_Ruan}
A.~Adem, J.~Leida, and Y.~Ruan, \emph{Orbifolds and stringy topology},
  Cambridge Tracts in Mathematics, vol. 171, Cambridge University Press,
  Cambridge, 2007.

\bibitem{Borzellino_Brunsden}
J.~Borzellino and V.~Brunsden, \emph{A manifold structure for the group of
  orbifold diffeomorphisms of a smooth orbifold}, J. Lie Theory \textbf{18}
  (2008), no.~4, 979--1007.

\bibitem{Chen_Ruan}
W.~Chen and Y.~Ruan, \emph{Orbifold {G}romov-{W}itten theory}, Orbifolds in
  mathematics and physics (Madison, WI, 2001), Contemp. Math., vol. 310, Amer.
  Math. Soc., Providence, RI, 2002, pp.~25--85.

\bibitem{Haefliger_orbifold}
A.~Haefliger, \emph{Groupo\"\i des d'holonomie et classifiants}, Ast\'erisque
  (1984), no.~116, 70--97.

\bibitem{Lerman_survey}
E.~{Lerman}, \emph{{Orbifolds as stacks?}}, {Enseign. Math. (2)} \textbf{56}
  (2010), no.~3-4, 315--363.

\bibitem{Lupercio_Uribe}
E.~Lupercio and B.~Uribe, \emph{Gerbes over orbifolds and twisted
  {$K$}-theory}, Comm. Math. Phys. \textbf{245} (2004), no.~3, 449--489.

\bibitem{Moerdijk_survey}
I.~Moerdijk, \emph{Orbifolds as groupoids: an introduction}, Orbifolds in
  mathematics and physics ({M}adison, {WI}, 2001), Contemp. Math., vol. 310,
  Amer. Math. Soc., Providence, RI, 2002, pp.~205--222.

\bibitem{Moerdijk_Mrcun}
I.~Moerdijk and J.~Mr\v{c}un, \emph{Introduction to foliations and {L}ie
  groupoids}, Cambridge Studies in Advanced Mathematics, vol.~91, Cambridge
  University Press, Cambridge, 2003.

\bibitem{Moerdijk_Pronk}
I.~Moerdijk and D.~Pronk, \emph{Orbifolds, sheaves and groupoids}, K-Theory
  \textbf{12} (1997), no.~1, 3--21.

\bibitem{Pronk}
D.~Pronk, \emph{Etendues and stacks as bicategories of fractions}, Compositio
  Math. \textbf{102} (1996), no.~3, 243--303.

\bibitem{Satake1}
I.~Satake, \emph{On a generalization of the notion of manifold}, Proc. Nat.
  Acad. Sci. U.S.A. \textbf{42} (1956), 359--363.

\bibitem{Satake2}
\bysame, \emph{The {G}auss-{B}onnet theorem for {$V$}-manifolds}, J. Math. Soc.
  Japan \textbf{9} (1957), 464--492.

\bibitem{Schmeding}
A.~{Schmeding}, \emph{{The diffeomorphism group of a non-compact orbifold.}},
  {Diss. Math.} \textbf{507} (2015), 179.

\bibitem{Thurston}
W.~Thurston, \emph{The geometry and topology of three-manifolds}, available at
  www.msri.org/publications/books/gt3m/.

\end{thebibliography}

\providecommand{\bysame}{\leavevmode\hbox to3em{\hrulefill}\thinspace}
\providecommand{\MR}{\relax\ifhmode\unskip\space\fi MR }
\providecommand{\MRhref}[2]{%
  \href{http://www.ams.org/mathscinet-getitem?mr=#1}{#2}
}
\providecommand{\href}[2]{#2}

\end{document}